\documentclass[12pt]{amsart}

\usepackage{bold-extra}
\usepackage{tikz-cd}
\usepackage{amsfonts}
\usepackage{amsthm}
\usepackage{xfrac}
\usepackage{faktor}
\usepackage{amssymb}
\usepackage{enumitem}
\usepackage{microtype}
\usepackage{mathrsfs}

\title[Gorenstein modules over large families of infinite groups]{Gorenstein modules and dimension over large families of infinite groups}
\author{Dimitra-Dionysia Stergiopoulou}

\newtheorem{Lemma}{Lemma}[section]
\newtheorem{Proposition}[Lemma]{Proposition}
\newtheorem{Theorem}[Lemma]{Theorem}
\newtheorem{Corollary}[Lemma]{Corollary}
\newtheorem{Remark}[Lemma]{Remark}
\newtheorem{Definition}[Lemma]{Definition}
\newtheorem{Question}[Lemma]{Question}

\DeclareMathOperator{\sfli}{sfli}

\begin{document}

\begin{abstract} 
 We give characterizations of Gorenstein projective, Gorenstein flat and Gorenstein injective modules over the group algebra for large families of infinite groups and show that every weak Gorenstein projective, weak Gorenstein flat and weak Gorenstein injective module is Gorenstein projective, Gorenstein flat and Gorenstein injective, respectively. These characterizations provide Gorenstein analogues of Benson's cofibrant modules. We deduce that, over a commutative ring of finite Gorenstein weak global dimension, every Gorenstein projective module is Gorenstein flat. Moreover, we study cases where the tensor product and the group of homomorphisms between modules over the group algebra is a Gorenstein module. Finally, we determine the Gorenstein homological dimension of an $\textsc{\textbf{lh}}\mathfrak{F}$-group over a commutative ring of finite Gorenstein weak global dimension.
\end{abstract}

\maketitle

\section{Introduction}Gorenstein homological algebra is the relative homological algebra which is based on Gorenstein projective, Gorenstein injective and Gorenstein flat modules. The standard reference for these modules and for the relative dimensions which are based on them is \cite{H1}. Recently, Saroch and Stovicek \cite{SS} introduced the notion of projectively coresolved Gorenstein flat modules (PGF modules, for short). Over a ring $R$, these modules are the syzygies of the acyclic complexes of projective modules that remain acyclic after applying the functor $I\otimes_R  \_\!\_$ for every injective module $I$. It is clear that PGF modules are Gorenstein flat. While in classical homological algebra every projective module is flat, the relation between Gorenstein projective and Gorenstein flat modules is still mysterious. As shown in \cite[Theorem 4.4]{SS}, every PGF module is Gorenstein projective. A natural question is whether the class of Gorenstein projective modules is contained in the class of PGF modules. 

In this paper, we study Gorenstein projective, Gorenstein flat and Gorenstein injective modules over large families of infinite groups. We assume first that the commutative ring $R$ is of finite Gorenstein weak global dimension and that $G$ is an $\textsc{\textbf{lh}}\mathfrak{X}$-group, where $\mathfrak{X}$ is the class of groups admitting a weak characteristic module. We note that the class $\textsc{\textbf{lh}}\mathfrak{X}$ contains the class of $\textsc{\textbf{lh}}\mathfrak{F}$-groups and the class of groups of type $\Phi_R$. We prove that every weak Gorenstein flat module is Gorenstein flat and every weak Gorenstein projective module is PGF.  Moreover, we give characterizations of Gorenstein projective, Gorenstein flat, PGF and Gorenstein injective modules in terms of the module $B(G,R)$ of all bounded $R$-valued functions on $G$, which provide Gorenstein analogues of Benson's cofibrant modules (see Theorems \ref{cora} and \ref{Cora}). 
\begin{Theorem}Let $R$ be a commutative ring of finite Gorenstein weak global dimension and $G$ be an $\textsc{\textbf{lh}}\mathfrak{X}$-group. Then:
	\begin{itemize}
		\item[(i)] ${\tt GFlat}(RG)={\tt WGFlat}(RG)=\mathscr{X}_{B,{\tt GFlat}}$, where $\mathscr{X}_{B,{\tt GFlat}}=\{M\in \textrm{Mod}(RG): M\otimes_R B(G,R)\in {\tt GFlat}(RG)\}$.
		\item[(ii)] $\mathscr{X}_{B,{\tt PGF}}={\tt PGF}(RG)={\tt WGProj}(RG)={\tt GProj}(RG)$, where  $\mathscr{X}_{B,{\tt PGF}}=\{M\in \textrm{Mod}(RG): M\otimes_R B(G,R)\in {\tt PGF}(RG)\}$.
	\end{itemize}
\end{Theorem}

As a result, we generalize \cite[Theorem 8.4]{St1} over any commutative ring $R$ of finite Gorenstein weak global dimension (see Corollary \ref{nionia}). 

\begin{Theorem} Let $R$ be a commutative ring of finite Gorenstein weak global dimension and $G$ be an $\textsc{\textbf{lh}}\mathfrak{X}$-group. Then, every Gorenstein projective $RG$-module is Gorenstein flat.
\end{Theorem}

Furthermore, denoting by $\mathfrak{Y}$ the class of groups admitting a characteristic module, we prove analogous results for Gorenstein injective $RG$-modules over any commutative ring $R$ of finite Gorenstein global dimension and any $\textsc{\textbf{h}}\mathfrak{Y}$-group $G$ (see Theorem \ref{Iicora}).

\begin{Theorem}Let $R$ be a commutative ring of finite Gorenstein global dimension and $G$ be an $\textsc{\textbf{h}}\mathfrak{Y}$-group. Then, $\mathscr{Y}_{B,{\tt GInj}}={\tt GInj}(RG)={\tt WGInj}(RG)$, where $\mathscr{Y}_{B,{\tt GInj}}=\{M\in \textrm{Mod}(RG): \textrm{Hom}_R (B(G,R),M)\in {\tt GInj}(RG)\}$.
\end{Theorem}

One of the open problems of Gorenstein homological algebra asks whether the tensor product of Gorenstein projective modules (respectively, Gorenstein flat modules) is Gorenstein projective (respectively, Gorenstein flat). Similarly, it is an open problem whether the group of homomorphisms between a Gorenstein projective and a Gorenstein injective module is Gorenstein injective. Even in the special case of modules over group rings these problems are open. Consider a commutative ring $R$ and a group $G$. It follows easily from the definitions and the diagonal action of $G$ that for every projective (respectively, flat) $RG$-module $M$ and every $RG$-module $N$ which is Gorenstein projective (respectively, Gorenstein flat) as $R$-module, the diagonal $RG$-module $M\otimes_R N$ is Gorenstein projective (respectively, Gorenstein flat). Similarly, for every projective $RG$-module $N$ and every $RG$-module $M$ which is Gorenstein injective as $R$-module, the diagonal $RG$-module $\textrm{Hom}(N,M)$ is Gorenstein injective. However, switching the roles of $R$-modules and $RG$-modules gives an open problem.

The following results of this paper, which will be crucial in our characterizations in Theorems 1.1 and 1.3, give an answer to this question for the large families of infinite groups we study (see Proposition \ref{prop1}, Corollary \ref{porisma} and Proposition \ref{Iprop1}).

\begin{Theorem}Let $R$ be a commutative ring of finite Gorenstein weak global dimension and $G$ be an $\textsc{\textbf{lh}}\mathfrak{X}$-group. Then:
	\begin{itemize}
		\item[(i)] For every (weak) Gorenstein flat $RG$-module $M$ and every $RG$-module $N$ which is flat as $R$-module, the $RG$-module $M\otimes_R N$ is Gorenstein flat.
		\item[(ii)] For every (weak) Gorenstein projective $RG$-module $M$ and every $RG$-module $N$ which is projective as $R$-module, the $RG$-module $M\otimes_R N$ is Gorenstein projective.
	\end{itemize}
\end{Theorem}

\begin{Theorem}Let $R$ be a commutative ring of finite Gorenstein global dimension and $G$ be an $\textsc{\textbf{h}}\mathfrak{Y}$-group. Then, for every (weak) Gorenstein injective $RG$-module $M$ and every $RG$-module $N$ which is projective as $R$-module, the $RG$-module $\textrm{Hom}_{R}(N,M)$ is Gorenstein injective.
\end{Theorem}

Furthermore, we provide a flat analogue of \cite[Proposition 7.4]{St1} which generalizes \cite[Theorem C]{Kr2} (see Corollary \ref{cor226}).

\begin{Theorem}Let $R$ be a commutative ring of finite Gorenstein weak global dimension and $G$ be an $\textsc{\textbf{lh}}\mathfrak{F}$-group. Then, $\textrm{f.k}(RG)=\textrm{sfli}(RG)=\textrm{fin.f.dim}(RG)$.
\end{Theorem}
This result is noteworthy, since it is not known whether all Gorenstein projective modules are Gorenstein flat over an arbitrary ring.

Finally, we determine the Gorenstein homological dimension ${{\textrm{Ghd}}_{R}G}$ of an $\textsc{\textbf{lh}}\mathfrak{F}$-group $G$ over any commutative ring $R$ of finite Gorenstein weak global dimension (see Theorem \ref{theo712})

\begin{Theorem}We have ${{\textrm{Ghd}}_{R}G}=\textrm{fd}_{RG} B(G,R)$ for every commutative ring of finite Gorenstein weak global dimension and every $\textsc{\textbf{lh}}\mathfrak{F}$-group $G$.
\end{Theorem}
The Gorenstein cohomological dimension $\textrm{Gcd}_{R}G$ of an $\textsc{\textbf{lh}}\mathfrak{F}$-group $G$ over a commutative ring $R$ of finite global dimension was studied in \cite[Theorem 3.1]{Bis2}, \cite[Theorem A.1]{ET2} and \cite[Theorem 7.6]{St1}.

The structure of the paper is as follows. In Section 2, we establish notation, terminology and preliminary results that will be used in the sequel. In Sections 3 and 4 we consider a commutative ring $R$ of finite Gorenstein weak global dimension and $\textsc{\textbf{lh}}\mathfrak{X}$-group $G$. We first prove Theorem 1.3 and then use the stability properties of the classes of Gorenstein flat and PGF modules (see Subsection 2.6) to obtain Theorems 1.1 and 1.2. We deduce that Gorenstein flatness and Gorenstein projectivity is closed under subgroups that are in $\textsc{\textbf{lh}}\mathfrak{X}$ (see Corollaries 3.7 and 4.11). 

In Section 5 we consider a commutative ring of finite Gorenstein global dimension and an $\textsc{\textbf{h}}\mathfrak{Y}$-group $G$. We first prove Theorem 1.5 and then use the stability property of the classes of Gorenstein flat and PGF modules (see Subsection 2.6) to obtain Theorem 1.3. We deduce that Gorenstein injectivity is closed under subgroups that are in $\textsc{\textbf{h}}\mathfrak{Y}$ (see Corollary 5.7). 

In our final section, we consider a commutative ring of finite Gorenstein weak global dimension and an $\textsc{\textbf{lh}}\mathfrak{F}$-group $G$. We first prove Theorem 1.6 and then use this result in the proof of Theorem 1.7. As a consequence, we obtain results about $\textsc{\textbf{lh}}\mathfrak{F}$-group of type FP$_{\infty}$ (see Corollaries 6.11, 6.12 and 6.13).

Some of our methods are inspired from methods already developed in the literature. More precisely, we apply ideas from the proofs of \cite[Theorem 3.5]{Kr1} \cite[Theorem C]{Kr2} \cite[Proposition 1.4]{Emtata} \cite[Propositions 2.5 and 2.25]{St1} \cite[Proposition 3.4]{St} in order to prove Propositions \ref{theor2}, \ref{Ppst} and \ref{Itheor2} (and also Remarks \ref{rem38}, \ref{rem46} and \ref{rem58}) and from the proofs of \cite[Theorem B]{DT} \cite[Theorem 1.2]{BS} in order to prove Propositions \ref{prop1}, \ref{Prop1} and \ref{Iprop1}. Useful tools in our proofs will also be the inequalities $\textrm{max}\{\textrm{Ghd}_R G,\sfli(R)\}\leq\textrm{sfli}(RG)\leq \textrm{Ghd}_R G + \textrm{sfli}(R)$ and $\textrm{max}\{\textrm{Gcd}_R G,\textrm{spli}(R)\}\leq\textrm{spli}(RG)\leq \textrm{Gcd}_R G + \textrm{spli}(R)$, proved in our recent work \cite{St,St1} (see \cite[Corollary 6.2]{St} and \cite[Corollary 2.18]{St1}, respectively).

\vspace{0.1in}
\noindent
{\em Conventions.}
All rings are assumed to be associative and unital and all ring homomorphisms will be unit preserving. Unless otherwise specified, all modules will be left $R$-modules. We will denote by $\textrm{Mod}(R)$ the category of $R$-modules and by ${\tt Proj}(R)$, ${\tt Flat}(R)$, ${\tt Inj}(R)$ the categories of projective, flat and injective $R$-modules, respectively.

\section{Preliminaries}
In this section, we collect certain notions and preliminary results that will be used in the sequel. 
\subsection{Gorenstein modules.}
An acyclic complex $\textbf{P}$ of projective modules is said to be a totally acyclic complex of projective modules if the complex of abelian groups $\mbox{Hom}_R(\textbf{P},Q)$
is acyclic for every projective module $Q$. Then, a module is Gorenstein 
projective if it is a syzygy of a complete projective resolution. We let ${\tt GProj}(R)$ 
be the class of Gorenstein projective modules. The Gorenstein projective dimension $\mbox{Gpd}_RM$ of a module $M$ is the length of a shortest resolution of $M$ by Gorenstein projective modules. If no such resolution of finite length exists, then we write $\mbox{Gpd}_RM = \infty$ (see \cite{H1}).

An acyclic complex $\textbf{F}$ of flat modules is said to be a totally acyclic complex of flat modules if the complex of abelian groups $I \otimes_R \textbf{F}$ is acyclic 
for every injective right module $I$. Then, a module is Gorenstein 
flat if it is a syzygy of a complete flat resolution. We let ${\tt GFlat}(R)$ 
be the class of Gorenstein flat modules. The Gorenstein flat dimension $\mbox{Gfd}_RM$ of a module $M$ is the length of a shortest resolution of $M$ by Gorenstein flat modules. If no such resolution of finite length exists, then we write $\mbox{Gfd}_RM = \infty$ (see \cite{Ben} and \cite[Corollary 4.12]{SS}).

An acyclic complex $\textbf{I}$ of injective modules is said to be a totally acyclic complex of injective modules if the complex of abelian groups $\textrm{Hom}_R(J,\textbf{I})$ is acyclic 
for every injective right module $I$. Then, a module is Gorenstein 
injective if it is a syzygy of a complete injective resolution. We let ${\tt GInj}(R)$ 
be the class of Gorenstein injective modules. The Gorenstein injective dimension $\mbox{Gid}_RM$ of a module $M$ is the length of a shortest (right) resolution of $M$ by Gorenstein injective modules. If no such resolution of finite length exists, then we write $\mbox{Gid}_RM = \infty$ (see \cite{H1}).

The projectively coresolved Gorenstein flat modules (PGF-modules, for short) were introduced by Saroch and Stovicek \cite{SS}. Such a module is a syzygy of an acyclic complex of projective modules $\textbf{P}$, which is such that the complex of abelian groups $I \otimes_R \textbf{P}$ is acyclic for every injective module $I$. It is clear that the class ${\tt PGF}(R)$ of PGF modules is contained in ${\tt GFlat}(R)$. The inclusion ${\tt PGF}(R) \subseteq {\tt GProj}(R)$ is proved in \cite[Theorem 4.4]{SS}. Moreover, the class of PGF $R$-modules is closed under extensions, direct sums, direct summands and kernels of epimorphisms. The PGF dimension $\mbox{PGF-dim}_RM$ of a module $M$ is the length of a shortest resolution of $M$ by PGF modules. If no such resolution of finite length exists, then we write $\mbox{PGF-dim}_RM = \infty$ (see \cite{DE}).

\subsection{Group rings.}Let $R$ be a commutative ring, $G$ be a group and consider the associated group ring $RG$. The standard reference for group cohomology is \cite{Br}. The anti-isomorphism of $RG$ which is induced by the map $g \rightarrow g^{-1}$, $g\in G$, enables us to view every right $RG$-module $M$ as a left $RG$-module $M$, and hence $RG\cong {(RG)}^{\textrm{op}}$. Using the diagonal action of the group $G$, the tensor product $M\otimes_R N$ of two $RG$-modules is also an $RG$-module using the diagonal action of $G$; we define $g\cdot (x \otimes y)=gx \otimes gy \in M\otimes_R N$ for every $g\in G$, $x\in M$ and $y\in N$. We note that for every projective $RG$-module $M$ and every $R$-projective $RG$-module $N$, the diagonal $RG$-module $M\otimes_R N$ is also projective. Similarly, for every flat $RG$-module $M$ and every $R$-flat $RG$-module $N$, the diagonal $RG$-module $M\otimes_R N$ is also flat.  We consider again two $RG$-modules $M,N$ and the $R$-module $\textrm{Hom}_R(M,N)$. Then, $G$ acts diagonally on $\textrm{Hom}_R(M,N)$; we define $(g\cdot f) (x)=gf(g^{-1}x)\in N$ for every $g\in G$, $f\in \textrm{Hom}_R(M,N)$ and $x\in M$. Moreover, for every $R$-projective $RG$-module $M$ and every injective $RG$-module $N$, the diagonal $RG$-module $\textrm{Hom}_R(M,N)$ is also injective.

\subsection{Gedrich-Gruenberg invariants and Gorenstein global dimensions}The invariants $\textrm{silp}(R)$, $\textrm{spli}(R)$ were defined by Gedrich and Gruenberg in \cite{GG} as the supremum of the injective lengths (dimensions) of projective modules and the supremum of the projective lengths (dimensions) of injective modules, respectively. The invariant $\textrm{sfli}(R)$ is defined similarly as the supremum of the flat lengths (dimensions) of injective modules. Since projective modules are flat, the inequality $\textrm{sfli}(R)\leq \textrm{spli}(R)$ is clear. Moreover, for every commutative ring $R$ we have the inequality $\textrm{silp}(R)\leq\textrm{spli}(R)$, with equality if $\textrm{spli}(R)<\infty$ (see \cite[Corollary 24]{DE}). Thus, for every commutative ring $R$, in view of \cite[Theorem 4.1]{Emm3}, the finiteness of $\textrm{spli}(R)$ is equivalent to the finiteness of $\textrm{Ggl.dim}R$, and then $\textrm{Ggl.dim}R=\textrm{spli}(R)$. Furthermore, for every commutative ring $R$, in view of \cite[Theorem 2.4]{CET}, the finiteness of $\textrm{sfli}(R)$ is equivalent to the finiteness of $\textrm{Gwgl.dim}R$, and then $\textrm{Gwgl.dim}R=\textrm{sfli}(R)$.

\subsection{(Weak) characteristic modules}Let $R$ be a commutative ring and $G$ be a group. We define a weak characteristic module for $G$ over $R$ as an $R$-flat $RG$-module $A$ with $\textrm{fd}_{RG}A <\infty$, which admits an $R$-pure $RG$-linear monomorphism $\jmath: R \rightarrow A$. We note that the existence of a weak characteristic module is equivalent to the existence of an $R$-projective $RG$-module $A'$ with $\textrm{fd}_{RG}A <\infty$, which admits an $R$-split $RG$-linear monomorphism $\jmath': R \rightarrow A'$ (see \cite[Theorem 5.10]{St}). If $\textrm{sfli}(R)<\infty$, the existence of a weak characteristic module for $G$ over $R$ is equivalent to the finiteness of $\textrm{sfli}(RG)$ (see \cite[Theorem 5.10]{St}). Weak characteristic modules generalize the notion of characteristic modules, i.e. an $R$-projective $RG$-module $\Lambda$ with $\textrm{pd}_{RG}A <\infty$, which admits an $R$-split $RG$-linear monomorphism $\iota: R \rightarrow \Lambda$. Characteristic modules were used to prove many properties of the Gorenstein cohomological dimension $\textrm{Gcd}_R G$ of a group $G$ (see \cite{BDT,Tal}). In \cite{St1,St}, (weak) characteristic modules where used to prove the inequalities $\textrm{sfli}(RG)\leq \textrm{Ghd}_R G + \textrm{sfli}(R)$ and $\textrm{spli}(RG)\leq \textrm{Gcd}_R G + \textrm{spli}(R)$.
\subsection{Large families of infinite groups}(see \cite{Kr,MS}) Let $\mathcal{X}$ be any class of groups. We define $\textsc{\textbf{h}}_0\mathcal{X}:=\mathcal{X}$, and for every ordinal number $\alpha>0$, we say that a group $G$ belongs to the class $\textsc{\textbf{h}}_{\alpha}\mathcal{X}$ if and only if there exists a finite dimensional contractible CW-complex on which $G$ acts such that every isotropy subgroup of the action belongs to $\textsc{\textbf{h}}_{\beta}\mathcal{X}$ for some ordinal $\beta < \alpha$. We say that a group $G$ belongs to $\textsc{\textbf{h}}\mathcal{X}$ if and 
only if there is an ordinal $\alpha$ such that $G$ belongs to $\textsc{\textbf{h}}_{\alpha}\mathcal{X}$. Moreover, we define a group $G$ to be in  $\textsc{\textbf{lh}}\mathcal{X}$ if and only if all finitely generated subgroups of $G$ are in $\textsc{\textbf{h}}\mathcal{X}$. Throughout this paper, we will denote by $\mathfrak{F}$ the class of finite groups, by $\mathfrak{X}$ the class of groups admitting a weak characteristic module and by $\mathfrak{Y}$ the subclass of $\mathfrak{X}$ consisting of groups admitting a characteristic module. All soluble groups, all groups of finite virtual cohomological dimension and all automorphism groups of Noetherian modules over a commutative ring are $\textsc{\textbf{lh}}\mathfrak{F}$-groups. The class $\textsc{\textbf{lh}}\mathfrak{F}$
is closed under extensions, ascending unions, free products with amalgamation and HNN
extensions.

A group $G$ is said to be of type $\Phi_R$ if it has the property that for every $RG$-module $M$, $\textrm{pd}_{RG}M<\infty$ if and only if $\textrm{pd}_{RH}M<\infty$ for every finite subgroup $H$ of $G$. These groups were defined over $\mathbb{Z}$ in \cite{Ta}. Over a commutative ring $R$ of finite global dimension, every group of finite virtual cohomological dimension and every group which acts on a tree with finite stabilizers is of type $\Phi_R$ (\cite[Corollary 2.6]{MS}). We denote by $\mathfrak{F}_{\phi}$ the class of groups of type $\Phi_R$.

Let $B(G,R)$ be the $RG$-module which consists of all functions from $G$ to $R$ whose image is a finite subset of $R$. The $RG$-module $B(G,R)$ is $R$-free and $RH$-free for every finite subgroup $H$ of $G$. For every element $\lambda \in R$, the constant function $\iota(\lambda)\in B(G,R)$ with value $\lambda$ is invariant under the action of $G$. The map $\iota: R \rightarrow B(G,R)$ which is defined in this way is then $RG$-linear and $R$-split. Indeed, for every fixed element $g\in G$, there exists an $R$-linear splitting for $\iota$ by evaluating functions at $g$. Moreover, the cokernel $\overline{B}(G,R)$ of $\iota$ is $R$-free (see \cite[Lemma 3.3]{Kr2} and \cite[Lemma 3.4]{BC}). We note that the $RG$-module $B(G,R)$ is a candidate for a (weak) characteristic module for any group $G$ over any commutative ring $R$. Furthermore, we have $\textrm{pd}_{RG}B(G,R)<\infty$ over any group $G$ of type $\Phi_R$ and hence $B(G,R)$ is a characteristic module for $G$. As every finite group is of type $\Phi_R$, it follows that $\mathfrak{F} \subseteq \mathfrak{F}_{\phi} \subseteq \mathfrak{Y} \subseteq \mathfrak{X}$ and hence the classes of groups $\mathfrak{F}_{\phi}, \textsc{\textbf{h}}\mathfrak{F}$ are contained in the class $\textsc{\textbf{h}}\mathfrak{Y}$ and the classes of groups $\mathfrak{F}_{\phi}, \textsc{\textbf{lh}}\mathfrak{F}$ are contained in the class $\textsc{\textbf{lh}}\mathfrak{X}$.

\begin{Remark}\rm Let $R$ be a commutative ring such that $\textrm{sfli}(R)<\infty$ and $G$ be a group such that $\textrm{pd}_{RG}B(G,R)<\infty$. Then, $\textrm{Gcd}_{R}G=\textrm{pd}_{RG}B(G,R)<\infty$.
	Indeed, the condition $\textrm{pd}_{RG}B(G,R)<\infty$ implies that $B(G,R)$ is a characteristic module for $G$ over $R$. Thus, the result follows from \cite[Corollary 2.9 and Proposition 2.21]{St1}. We note that this is a generalization of \cite[Theorem 1.18]{Bis2}.
\end{Remark}	

\subsection{Stabilities of Gorenstein modules} In our proofs we also make use of the stability properties of Gorenstein flat, PGF and Gorenstein injective modules established in \cite{BK}, \cite{KS} and \cite{Bouch}, respectively. 
\subsubsection{Stability of Gorenstein flat modules}(\cite[Theorem 1.2]{BK} and \cite[Corollary 4.12]{SS}) We denote by ${\tt GFlat}^{(2)}(R)$ (respectively, ${\tt GFlat}^{(2)}_{\mathcal{I}}(R)$ the class of $R$-modules $M$ which are syzygies in an acyclic complex $\mathbf{F}$ of Gorenstein flat modules such that the complex $H\otimes_R\mathbf{F}$ is acyclic for every Gorenstein injective $R$-module $H$ (respectively, the complex $I\otimes_R\mathbf{F}$ is acyclic for every injective $R$-module $I$). Then, ${\tt GFlat}(R)={\tt GFlat}^{(2)}(R)={\tt GFlat}^{(2)}_{\mathcal{I}}(R)$.
\subsubsection{Stability of PGF modules}(\cite[Theorem 5.7]{KS}) We denote by ${\tt PGF}^{(2)}(R)$ (respectively, ${\tt PGF}^{(2)}_{\mathcal{I}}(R)$) the class of $R$-modules $M$ which are syzygies in an acyclic complex $\mathbf{P}$ of PGF modules such that the complex $H\otimes_R\mathbf{P}$ is acyclic for every Gorenstein injective $R$-module $H$ (respectively, the complex $I\otimes_R\mathbf{P}$ is acyclic for every injective $R$-module $I$). Then, ${\tt PGF}(R)={\tt PGF}^{(2)}(R)={\tt PGF}^{(2)}_{\mathcal{I}}(R)$.
\subsubsection{Stability of Gorenstein injective modules}(\cite[Theorem 2.1]{Bouch}) We denote by ${\tt GInj}^{(2)}(R)$ (respectively, ${\tt GInj}^{(2)}_{\mathcal{I}}(R)$) the class of $R$-modules $M$ which are syzygies in an acyclic complex $\mathbf{J}$ of Gorenstein injective modules such that the complex $\textrm{Hom}_R(H,\mathbf{J})$ is acyclic for every Gorenstein injective $R$-module $H$ (respectively, the complex $\textrm{Hom}_R(I,\mathbf{J})$ is acyclic for every injective $R$-module $I$). Then, ${\tt GInj}(R)={\tt GInj}^{(2)}(R)={\tt GInj}^{(2)}_{\mathcal{I}}(R)$.


\subsection{Weak Gorenstein modules}Let $R$ be a commutative ring. We denote by ${\tt WGProj}(R)$ the class of modules which are syzygies of an acyclic complex of projective modules $\textbf{P}$. We note that ${\tt GProj}(R)\subseteq {\tt WGProj}(R)$ and ${\tt PGF}(R)\subseteq {\tt WGProj}(R)$. Since the finiteness of $\textrm{sfli}(R)$ yields ${\tt WGProj}(R)\subseteq {\tt PGF}(R)$, we have ${\tt WGProj}(R)={\tt PGF}(R)={\tt GProj}(R)$ (see \cite[Theorem 4.4]{SS}).

Analogously, we denote by ${\tt WGFlat}(R)$ the class of modules which are syzygies of an acyclic complex of flat modules $\textbf{F}$. We note that ${\tt GFlat}(R)\subseteq {\tt WGFlat}(R)$. Moreover, the finiteness of $\textrm{sfli}(R)$ implies that ${\tt WGFlat}(R)\subseteq {\tt GFlat}(R)$, and hence ${\tt WGFlat}(R)={\tt GFlat}(R)$.

Finally, we denote by ${\tt WGInj}(R)$ the class of modules which are syzygies of an acyclic complex of injective modules $\textbf{I}$. We note that ${\tt GInj}(R)\subseteq {\tt WGInj}(R)$. Moreover, the finiteness of $\textrm{spli}(R)$ implies that ${\tt WGInj}(R)\subseteq {\tt GInj}(R)$, and hence ${\tt WGInj}(R)={\tt GInj}(R)$.

\section{Gorenstein flat modules over $\textsc{\textbf{lh}}\mathfrak{X}$-groups} We consider a commutative ring $R$ such that $\textrm{sfli}(R)<\infty$ and an $\textsc{\textbf{lh}}\mathfrak{X}$-group $G$. Our goal in this section is to give a characterization of the class ${\tt GFlat}(RG)$, in terms of the $RG$-module $B(G,R)$. Moreover, under these conditions, we conclude that the class ${\tt GFlat}(RG)$ coincides with the class ${\tt WGFlat}(RG)$. 

\begin{Lemma}\label{llem45}	Let $R$ be a commutative ring, $G$ be a group and $H$ be a subgroup of $G$. Then, for every Gorenstein flat $RH$-module $M$, the $RG$-module $\textrm{Ind}^G_H M$ is also Gorenstein flat.
\end{Lemma}

\begin{proof} Let $M$ be a Gorenstein flat $RH$-module. Then, there exists an acyclic complex of flat $RH$-modules $$\textbf{F}=\cdots \rightarrow F_{2}\rightarrow F_1\rightarrow F_0 \rightarrow F_{-1}\rightarrow \cdots$$ such that $M=\textrm{Im}(F_1 \rightarrow F_0)$ and the complex $I\otimes_{RH}\textbf{F}$ is exact, whenever $I$ is an injective $RH$-module. Thus, the induced complex $$\textrm{Ind}^G_H\textbf{F}=\cdots \rightarrow\textrm{Ind}^G_H F_2 \rightarrow\textrm{Ind}^G_H F_1\rightarrow\textrm{Ind}^G_H F_0 \rightarrow\textrm{Ind}^G_H F_{-1}\rightarrow \cdots$$ is an acyclic complex of flat $RG$-modules and has the $RG$-module $\textrm{Ind}^G_H M$ as syzygy. Since every injective $RG$-module $I$ is restricted to an injective $RH$-module, the isomorphism of complexes $I\otimes_{RG}\textrm{Ind}^G_H \textbf{F} \cong I\otimes_{RH}\textbf{F}$ implies that the $RG$-module $\textrm{Ind}^G_H M$ is Gorenstein flat.
\end{proof}

\begin{Proposition}\label{prop1}Let $R$ be a commutative ring such that $\textrm{sfli}(R)<\infty$ and $G$ be an $\textsc{\textbf{lh}}\mathfrak{X}$-group. Consider a weak Gorenstein flat $RG$-module $M$ and an $RG$-module $N$ which is flat as $R$-module. Then, $M\otimes_R N \in{\tt G Flat}(RG)$.
\end{Proposition}

\begin{proof}Let $M\in{\tt WGFlat}(RG)$ and $N\in{\tt Flat}(R)$. We will first show that $M\otimes_R N$ is Gorenstein flat as $RH$-module over any $\textsc{\textbf{h}}\mathfrak{X}$-subgroup $H$ of $G$. We use transfinite induction on the ordinal number $\alpha$, which is such that $H\in \textsc{\textbf{h}}_{\alpha}\mathfrak{X}$. If $\alpha=0$, then $H$ admits a weak characteristic module, and hence \cite[Theorem 3.14]{St} yields $\textrm{sfli}(RH) <\infty$. Since $M$ is a weak Gorenstein flat $RG$-module, it follows that $M$ is also a weak Gorenstein flat $RH$-module and hence there exists an acyclic complex of flat $RH$-modules $$\textbf{F}=\cdots \rightarrow F_{2}\rightarrow F_1\rightarrow F_0 \rightarrow F_{-1}\rightarrow \cdots$$ such that $M=\textrm{Im}(F_1 \rightarrow F_0)$. As $N$ is $R$-flat, we obtain the induced complex of $RH$-flat modules (with diagonal action) $$\textbf{F}\otimes_R N = \cdots \rightarrow F_{2}\otimes_R N\rightarrow F_1\otimes_R N\rightarrow F_0\otimes_R N \rightarrow F_{-1}\otimes_R N\rightarrow \cdots,$$ where $M\otimes_R N= \textrm{Im}(F_1\otimes_R N \rightarrow F_0\otimes_R N)$. Then, the finiteness of $\textrm{sfli}(RH)$ implies that the complex $I\otimes_{RH}(\textbf{F}\otimes_R N)$ is acyclic for every injective $RH$-module $I$. We conclude that the $RH$-module $M\otimes_R N$ is Gorenstein flat.  
	
	Now we assume that $M\otimes_R N$ is Gorenstein flat as $RH'$-module for every $\textsc{\textbf{h}}_{\beta}\mathfrak{X}$-subgroup $H'$ of $G$ and every $\beta<\alpha$. Let $H$ be an $\textsc{\textbf{h}}_{\alpha}\mathfrak{X}$-subgroup of $G$. Then, there
	exists a $\mathbb{Z}$-split exact sequence of $\mathbb{Z}H$-modules $$0\rightarrow C_r \rightarrow \cdots \rightarrow C_1 \rightarrow C_0 \rightarrow \mathbb{Z} \rightarrow 0,$$	where each $C_i$
	is a direct sum of permutation $\mathbb{Z}H$-modules of the form $\mathbb{Z}[H/H']$, with $H'$ an $\textsc{\textbf{h}}_{\beta}\mathfrak{X}$-subgroup of $H$ for some $\beta<\alpha$. We note that the integer $r$ is the dimension of the $H$-CW-complex provided by the definition of $H$ being an $\textsc{\textbf{h}}_{\alpha}\mathfrak{X}$-group. The above exact sequence yields an exact sequence of $RH$-modules of the form
	\begin{equation}\label{eqqqq}
		0\rightarrow K_r \rightarrow \cdots \rightarrow K_1 \rightarrow K_0 \rightarrow M\otimes_R N \rightarrow 0
	\end{equation}
	such that every $K_i$ is a direct sum of modules of the form ${\textrm{Ind}^H_{H'}}{\textrm{Res}^H_{H'}} (M \otimes_R N)$, where $H'\in \textsc{\textbf{h}}_{\beta}\mathfrak{X}$, $\beta<\alpha$ (see also \cite[Lemma 2.3]{Bis2}).
	Our induction hypothesis implies that ${\textrm{Res}^H_{H'}} (M \otimes_R N)$ is a Gorenstein flat $RH'$-module. Invoking Lemma \ref{llem45}, we infer that ${\textrm{Ind}^H_{H'}}{\textrm{Res}^H_{H'}}(M \otimes_R N)$ is a Gorenstein flat $RH$-module. Since the class ${\tt GFlat}(RH)$ is closed under direct sums, we obtain that the $RH$-module $K_i$ is Gorenstein flat, for every $i=0,\dots r$. Thus, the exact sequence (\ref{eqqqq}) yields $\textrm{Gfd}_{RH}(M\otimes_R N)\leq r$. Moreover, $M\in {\tt WGFlat}(RG)$, and hence there exists an exact sequence of $RG$-modules of the form $$0\rightarrow M \rightarrow F_{r-1} \rightarrow \cdots \rightarrow F_1 \rightarrow F_0 \rightarrow M' \rightarrow 0,$$ where $F_i$ is flat for every $i=0,1,\dots,r-1$ and $M'\in {\tt WGFlat}(RG)$. Since $N$ is $R$-flat, we obtain the induced exact sequence of $RG$-modules (with diagonal action)
	\begin{equation}\label{eqqq}
		0\rightarrow M \otimes_R N\rightarrow F_{r-1}\otimes_R N \rightarrow \cdots \rightarrow F_0\otimes_R N \rightarrow M'\otimes_R N \rightarrow 0,
	\end{equation}where $F_i \otimes_R N$ is a flat $RG$-module (and hence is flat as $RH$-module) for every $i=0,1,\dots,r-1$. The same argument as above for the $RG$-module $M' \in {\tt WGflat}(RG)$ yields $\textrm{Gfd}_{RH}(M'\otimes_R N)\leq r$. Since every ring is ${\tt GF}$-closed, using \cite[Theorem 2.8]{Ben}, we conclude that $M\otimes_R N$ is a Gorenstein flat $RH$-module.
	
	Let $G$ be an $\textsc{\textbf{lh}}\mathfrak{X}$-group. Then, $G$ can be expressed as the filtered union of its finitely generated subgroups $(G_{\lambda})_{\lambda}$, which are all contained in $\textsc{\textbf{h}}\mathfrak{X}$. Since $G_{\lambda}\in \textsc{\textbf{h}}\mathfrak{X}$, the $RG_{\lambda}$-module $M\otimes_R N$ is Gorenstein flat. Invoking Lemma \ref{llem45}, we obtain that the $RG$-module $\textrm{Ind}^G_{G_{\lambda}}(M\otimes_R N)$ is Gorenstein flat as well. Thus, the $RG$-module $M\otimes_R N\cong  {\lim\limits_{\longrightarrow}}_{\lambda}\textrm{Ind}^G_{G_{\lambda}}(M\otimes_R N)$ is Gorenstein flat as a direct limit of Gorenstein flat modules (see \cite[Corollary 4.12]{SS}).\end{proof}

\begin{Definition}Let $R$ be a commutative ring and $G$ be a group. We denote by $\mathscr{X}_{B,{\tt GFlat}}$ the class of $RG$-modules $\mathscr{X}_{B,{\tt GFlat}}=\{M\in \textrm{Mod}(RG): \, M\otimes_R B(G,R)\in {\tt GFlat}(RG)\}$.
\end{Definition}

\begin{Corollary}\label{theor1}Let $R$ be a commutative ring such that $\textrm{sfli}(R)<\infty$ and $G$ be an $\textsc{\textbf{lh}}\mathfrak{X}$-group. Then, ${\tt WG Flat}(RG)\subseteq\mathscr{X}_{B,{\tt GFlat}}$.
\end{Corollary}

\begin{proof}Since the $RG$-module $B(G,R)$ is $R$-free, this is an immediate consequence of Proposition \ref{prop1}.
\end{proof}


\begin{Proposition}\label{theor2}Let $R$ be a commutative ring such that $\textrm{sfli}(R)<\infty$ and $G$ be an $\textsc{\textbf{lh}}\mathfrak{X}$-group. Then, $\mathscr{X}_{B,{\tt GFlat}}\subseteq{\tt GFlat}(RG)$.
\end{Proposition}

\begin{proof}Let $B=B(G,R)$, $\overline{B}=\overline{B}(G,R)$ and consider an $RG$-module $M$ such that the $RG$-module $M\otimes_R B$ is Gorenstein flat. We also let $V_i=\overline{B}^{\otimes i}\otimes_R B$ for every $i\geq 0$, where $\overline{B}^{\otimes 0}=R$. Since the short exact sequence of $RG$-modules $0\rightarrow R \rightarrow B \rightarrow \overline{B}\rightarrow 0$ is $R$-split, we obtain for every $i\geq 0$ a short exact sequence of $RG$-modules of the form $$0\rightarrow M\otimes_R\overline{B}^{\otimes i}\rightarrow M\otimes_R V_i \rightarrow M\otimes_R \overline{B}^{\otimes i+1}\rightarrow 0.$$ Then, the splicing of the above short exact sequences for every $i\geq 0$ yields an exact sequence of the form
	\begin{equation}\label{eq1}
		0\rightarrow M \xrightarrow{\alpha} M\otimes_R V_0 \rightarrow M\otimes_R V_1 \rightarrow M\otimes_R V_2 \rightarrow \cdots.
	\end{equation} 
	Since the $RG$-module $M\otimes_R B$ is Gorenstein flat and $\overline{B}$ is $R$-flat, we obtain that the $RG$-module $M\otimes_R V_i\cong (M\otimes_R B)\otimes_R \overline{B}^{\otimes i}$ is Gorenstein flat for every $i\geq 0$, by Proposition \ref{prop1}. We also consider an $RG$-flat resolution of $M$
	\begin{equation*}
		\textbf{Q}=\cdots \rightarrow Q_2 \rightarrow Q_1 \rightarrow Q_0 \xrightarrow{\beta} M \rightarrow 0.
	\end{equation*} 
	Splicing the resolution $\textbf{Q}$ with the exact sequence (\ref{eq1}), we obtain an acyclic complex of Gorenstein flat $RG$-modules 
	\begin{equation*}
		\mathfrak{P}=\cdots \rightarrow Q_2 \rightarrow Q_1 \rightarrow Q_0 \xrightarrow{\alpha \beta} M\otimes_R V_0 \rightarrow M\otimes_R V_1 \rightarrow M\otimes_R V_2 \rightarrow \cdots
	\end{equation*} 
	which has syzygy the $RG$-module $M$. It suffices to prove that the complex $I\otimes_{RG}\mathfrak{P}$ is acyclic for every injective $RG$-module $I$. Using \cite[Theorem 1.2]{BK} we will then obtain that the $RG$-module $M$ is Gorenstein flat. Let $I$ be an injective $RG$-module. Then, the $R$-split short exact sequence of $RG$-modules $0\rightarrow R \rightarrow B \rightarrow \overline{B}\rightarrow 0$ yields an induced exact sequence of $RG$-modules with diagonal action $0\rightarrow I\rightarrow B \otimes_{R}I\rightarrow \overline{B}\otimes_{R} I\rightarrow 0$ which is $RG$-split. Thus, it suffices to prove that the complex $(B\otimes_{R}I)\otimes_{RG}\mathfrak{P}$ is acyclic. Since $B$ is $R$-flat, we obtain that the acyclic complex $\textbf{Q}\otimes_R B$ is a flat resolution of the Gorenstein flat $RG$-module $M\otimes_{R}B$. Hence, every syzygy module of $\textbf{Q}\otimes_R B$ is also a Gorenstein flat $RG$-module (see \cite[Lemma 2.4]{Ben}). Moreover, the $RG$-module $(M\otimes_R B)\otimes_R \overline{B}^{\otimes i}\cong (M\otimes_R \overline{B}^{\otimes i})\otimes_R B$ is Gorenstein flat for every $i\geq 0$. Consequently, every syzygy module of the acyclic complex
	\begin{equation*}
		\mathfrak{P}\otimes_R B =\cdots\rightarrow  Q_1\otimes_R B\rightarrow  Q_0\otimes_R B \rightarrow  M\otimes_R V_0 \otimes_R B\rightarrow  M\otimes_R V_1 \otimes_R B\rightarrow \cdots
	\end{equation*} is a Gorenstein flat $RG$-module. Since the functor $\textrm{Tor}^{RG}_1 (I,\_\!\_)$ vanishes on Gorenstein flat $RG$-modules, we conclude that the complex $(B\otimes_{R}I)\otimes_{RG}\mathfrak{P}\cong I\otimes_{RG}(\mathfrak{P}\otimes_{R}B)$ is acyclic, as needed.\end{proof}

\begin{Theorem}\label{cora}Let $R$ be a commutative ring such that $\textrm{sfli}(R)<\infty$ and $G$ be an $\textsc{\textbf{lh}}\mathfrak{X}$-group. Then, $\mathscr{X}_{B,{\tt GFlat}}={\tt GFlat}(RG)={\tt WGFlat}(RG)$.
\end{Theorem}

\begin{proof}By Corollary \ref{theor1}, we have ${\tt WG Flat}(RG)\subseteq\mathscr{X}_{B,{\tt GFlat}}$. Moreover, Proposition \ref{theor2} yields $\mathscr{X}_{B,{\tt GFlat}}\subseteq{\tt GFlat}(RG)$ and the inclusion ${\tt GFlat}(RG)\subseteq{\tt WGFlat}(RG)$ is clear. We conclude that $\mathscr{X}_{B,{\tt GFlat}}={\tt GFlat}(RG)={\tt WGFlat}(RG)$, as needed.
\end{proof}


\begin{Corollary}Let $R$ be a commutative ring of finite Gorenstein weak global dimension. Then, Gorenstein flateness is closed under subgroups that are in $\textsc{\textbf{lh}}\mathfrak{X}$.
\end{Corollary}

\begin{proof}Let $M$ be a Gorenstein flat $RG$-module and $H$ be a subgroup of $G$ such that $H\in\textsc{\textbf{lh}}\mathfrak{X}$. Then, the $RH$-module $\textrm{Res}_H^G(M)$ is weak Gorenstein flat and hence Theorem \ref{cora} implies that $\textrm{Res}_H^G(M)\in{\tt GFlat}(RH)$.\end{proof}

\begin{Remark}\label{rem38}\rm Let $R$ be a commutative ring such that $\textrm{sfli}(R)<\infty$ and $G$ be a group in the class $\mathfrak{X}$, i.e. there exists a weak characteristic module $A$ for $G$ over $R$.  We also consider an $RG$-module $M$ such that the $RG$-module $M\otimes_R A$ is Gorenstein flat. Then, $M$ is a Gorenstein flat $RG$-module. 
	
	Indeed, there exists an $R$-pure $RG$-short exact sequence $0\rightarrow R \rightarrow A \rightarrow \overline{A}\rightarrow 0$, where the $RG$-modules $A,\overline{A}$ are $R$-flat. Following step by step the proof of Proposition \ref{theor2}, we construct an acyclic complex of Gorenstein flat modules 
	\begin{equation*}
		\mathfrak{P'}=\cdots \rightarrow Q_2' \rightarrow Q_1' \rightarrow Q_0' \rightarrow M\otimes_R V_0' \rightarrow M\otimes_R V_1' \rightarrow M\otimes_R V_2' \rightarrow \cdots,
	\end{equation*} where $V_i'={\overline{A}}^{\otimes i}\otimes_R A$, for every $i\geq 0$, and has syzygy the $RG$-module $M$. Using the $R$-pure $RG$-short exact sequence $0\rightarrow R \rightarrow A \rightarrow \overline{A} \rightarrow 0 $ and \cite[Theorem 1.2]{BK}, it suffices to show that the complex $I\otimes_{RG}(\mathfrak{P'}\otimes_{R}A)$ is acyclic for every injective $RG$-module $I$. This follows exactly as in the proof of Proposition \ref{theor2}, since every syzygy module of $\mathfrak{P'}\otimes_{R}A$ is Gorenstein flat.
\end{Remark}

\begin{Corollary} Let $R$ be a commutative ring of finite weak global dimension and $G$ be a group of type $\Phi_R$. Then, $\mathscr{X}_{B,{\tt Flat}}={\tt GFlat}(RG)={\tt WGFlat}(RG)$, where $\mathscr{X}_{B,{\tt Flat}}=\{M\in \textrm{Mod}(RG): \, M\otimes_R B(G,R)\in {\tt Flat}(RG)\}$.
\end{Corollary}

\begin{proof}In view of Theorem \ref{cora}, it suffices to show that $\mathscr{X}_{B,{\tt GFlat}}\subseteq\mathscr{X}_{B,{\tt Flat}}$. Let $M\in \mathscr{X}_{B,{\tt GFlat}}$. Then, $M\in {\tt WGFlat}(RG)\subseteq {\tt WGFlat}(R)$, and hence the finiteness of $\textrm{wgl.dim}(R)$ implies that $M$ is $R$-flat. Since $\textrm{fd}_{RG} B(G,R)<\infty$, we obtain that $\textrm{fd}_{RG}M\otimes_R B(G,R)<\infty$. We conclude that $M\otimes_R B(G,R)\in {\tt GFlat}(RG)\cap \overline{{\tt Flat}}(RG)={\tt Flat}(RG)$ (see \cite[Lemma 2.4]{KS}).
\end{proof}

\section{Gorenstein projective and PGF modules over $\textsc{\textbf{lh}}\mathfrak{X}$-groups}

We consider a commutative ring $R$ such that $\textrm{sfli}(R)<\infty$ and an $\textsc{\textbf{lh}}\mathfrak{X}$-group $G$. Our goal in this section is to give a characterization of the class ${\tt GProj}(RG)$ related to the $RG$-module $B(G,R)$. Moreover, under these conditions, we conclude that the classes ${\tt GProj}(RG)$, ${\tt PGF}(RG)$ and ${\tt WGProj}(RG)$ coincide. As a result, under the above conditions, every Gorenstein projective $RG$-module is Gorenstein flat.

\begin{Lemma}\label{Llem45}{\rm(\cite[Lemma 2.12]{St1})} Let $R$ be a commutative ring, $G$ be a group and $H$ be a subgroup of $G$. Then, for every PGF $RH$-module $M$, the $RG$-module $\textrm{Ind}^G_H M$ is also PGF.
\end{Lemma}

\begin{Proposition}\label{Prop1}Let $R$ be a commutative ring such that $\textrm{sfli}(R)<\infty$ and $G$ be an $\textsc{\textbf{lh}}\mathfrak{X}$-group. Consider a weak Gorenstein projective $RG$-module $M$ and an $RG$-module $N$ which is projective as $R$-module. Then, $M\otimes_R N \in{\tt PGF}(RG)$.
\end{Proposition}

\begin{proof}Let $M\in{\tt WGProj}(RG)$ and $N\in{\tt Proj}(R)$. We will first show that $M\otimes_R N$ is PGF as $RH$-module over any $\textsc{\textbf{h}}\mathfrak{X}$-subgroup $H$ of $G$. We use transfinite induction on the ordinal number $\alpha$, which is such that $H\in \textsc{\textbf{h}}_{\alpha}\mathfrak{X}$. If $\alpha=0$, then $H$ admits a weak characteristic module, and hence \cite[Theorem 3.14]{St} yields $\textrm{sfli}(RH) <\infty$. Since $M$ is a weak Gorenstein projective $RG$-module, it follows that $M$ is also a weak Gorenstein projective $RH$-module and hence there exists an acyclic complex of projective $RH$-modules $$\textbf{P}=\cdots \rightarrow P_{2}\rightarrow P_1\rightarrow P_0 \rightarrow P_{-1}\rightarrow \cdots$$ such that $M=\textrm{Im}(P_1 \rightarrow P_0)$. Since $N$ is $R$-projective, we obtain the induced complex of $RH$-projective modules (with diagonal action) $$\textbf{P}\otimes_R N = \cdots \rightarrow P_{2}\otimes_R N\rightarrow P_1\otimes_R N\rightarrow P_0\otimes_R N \rightarrow P_{-1}\otimes_R N\rightarrow \cdots,$$ where $M\otimes_R N= \textrm{Im}(P_1\otimes_R N \rightarrow P_0\otimes_R N)$. Then, the finiteness of $\textrm{sfli}(RH)$ implies that the complex $I\otimes_{RH}(\textbf{P}\otimes_R N)$ is acyclic for every injective $RH$-module $I$. We conclude that the $RH$-module $M\otimes_R N$ is PGF. 
	
	 Now we assume that $M\otimes_R N$ is PGF as $RH'$-module for every $\textsc{\textbf{h}}_{\beta}\mathfrak{X}$-subgroup $H'$ of $G$ and every $\beta<\alpha$. Let $H$ be an $\textsc{\textbf{h}}_{\alpha}\mathfrak{X}$-subgroup of $G$. Then, there
	exists an exact sequence of $\mathbb{Z}H$-modules $$0\rightarrow C_r \rightarrow \cdots \rightarrow C_1 \rightarrow C_0 \rightarrow \mathbb{Z} \rightarrow 0$$	where each $C_i$
	is a direct sum of permutation $\mathbb{Z}H$-modules of the form $\mathbb{Z}[H/H']$, with $H'$ an $\textsc{\textbf{h}}_{\beta}\mathfrak{X}$-subgroup of $H$ for some $\beta<\alpha$. We note that the integer $r$ is the dimension of the $H$-CW-complex provided by the definition of $H$ being an $\textsc{\textbf{h}}_{\alpha}\mathfrak{X}$-group. The above exact sequence yields an exact sequence of $RH$-modules of the form
	\begin{equation}\label{Eqqqq}
		0\rightarrow K_r \rightarrow \cdots \rightarrow K_1 \rightarrow K_0 \rightarrow M\otimes_R N \rightarrow 0,
	\end{equation}
	such that every $K_i$ is a direct sum of modules of the form ${\textrm{Ind}^H_{H'}}{\textrm{Res}^H_{H'}} (M \otimes_R N)$, where $H'\in \textsc{\textbf{h}}_{\beta}\mathfrak{X}$, $\beta<\alpha$ (see also \cite[Lemma 2.3]{Bis2}).
	Our induction hypothesis implies that ${\textrm{Res}^H_{H'}} (M \otimes_R N)$ is a PGF $RH'$-module. From Lemma \ref{Llem45} we deduce that ${\textrm{Ind}^H_{H'}}{\textrm{Res}^H_{H'}}(M \otimes_R N)$ is a PGF $RH$-module. The class ${\tt PGF}(RH)$ is closed under direct sums, and hence the $RH$-module $K_i$ is PGF, for every $i=0,\dots r$. Thus, the exact sequence (\ref{Eqqqq}) yields $\textrm{PGF-dim}_{RH}(M\otimes_R N)\leq r$. Moreover, $M\in {\tt WGProj}(RG)$, and hence there exists an exact sequence of $RG$-modules of the form $$0\rightarrow M \rightarrow P_{r-1} \rightarrow \cdots \rightarrow P_1 \rightarrow P_0 \rightarrow M' \rightarrow 0,$$ where $P_i$ is projective for every $i=0,1,\dots,r-1$ and $M'\in {\tt WGProj}(RG)$. Since $N$ is $R$-projective, we obtain the induced exact sequence of $RG$-modules (with diagonal action)
	\begin{equation}\label{Eqqq}
		0\rightarrow M \otimes_R N\rightarrow P_{r-1}\otimes_R N \rightarrow \cdots \rightarrow P_0\otimes_R N \rightarrow M'\otimes_R N \rightarrow 0,
	\end{equation}where $P_i \otimes_R N$ is a projective $RG$-module (and hence projective as $RH$-module) for every $i=0,1,\dots,r-1$. The same argument as above for the weak Gorenstein projective $RG$-module $M'$ shows that $\textrm{PGF-dim}_{RH}(M'\otimes_R N)\leq r$. In view of \cite[Proposition 2]{DE}, we conclude that $M\otimes_R N$ is a PGF $RH$-module.
	
	Let $G$ be an $\textsc{\textbf{lh}}\mathfrak{X}$-group. We will proceed by induction on the cardinality of $G$. If $G$ is a
	countable group, then $G$ acts on a tree with stabilizers that are certain finitely generated subgroups of $G$, and hence $G\in \textsc{\textbf{h}}\mathfrak{F}$ (see \cite[Lemma 2.5]{JKL}). Thus, we assume that $G$ is uncountable. The group $G$ may then be expressed as a continuous ascending union of subgroups $G = \cup_{\lambda<\delta} G_{\lambda}$, for some ordinal $\delta$, where each $G_{\lambda}$ has strictly smaller cardinality than $G$. By induction we have that $M\otimes_R N$ is PGF as $RG_{\lambda}$-module, for every $\lambda<\delta$. Thus, invoking \cite[Proposition 4.5]{St1}, we infer that $\textrm{PGF-dim}_{RG}(M\otimes_R N)\leq 1$. Since $M\in {\tt WGProj}(RG)$, there exists a short exact sequence of $RG$-modules of the form $$0\rightarrow M \rightarrow P \rightarrow M'' \rightarrow 0,$$ where $M''\in {\tt WGProj}(RG)$ and $P\in {\tt Proj}(RG)$. Since $N$ is $R$-projective, we obtain the following short exact sequence of $RG$-modules (with diagonal action) \begin{equation}\label{equu}0\rightarrow M\otimes_R N\rightarrow P \otimes_R N\rightarrow M''\otimes_R N\rightarrow 0,
	\end{equation} where the $RG$-module $P\otimes_R N$ is projective. The same argument as before for the $RG$-module $M'' \in {\tt WGProj}(RG)$ yields $\textrm{PGF-dim}_{RG}(M''\otimes_R N)\leq 1$, and hence the exact sequence (\ref{equu}) implies that the $RG$-module $M\otimes_R N$ is PGF, as needed.\end{proof}

\begin{Corollary}\label{porisma}Let $R$ be a commutative ring such that $\textrm{sfli}(R)<\infty$ and $G$ be an $\textsc{\textbf{lh}}\mathfrak{X}$-group. Consider a Gorenstein projective $RG$-module $M$ and an $RG$-module $N$ which is projective as $R$-module. Then, $M\otimes_R N \in{\tt GProj}(RG)$.
\end{Corollary}

\begin{Definition}Let $R$ be a commutative ring and $G$ be a group. We denote by $\mathscr{X}_{B,{\tt GProj}}$ the class of $RG$-modules $\mathscr{X}_{B,{\tt GProj}}=\{M\in \textrm{Mod}(RG): \, M\otimes_R B(G,R)\in {\tt GProj}(RG)\}$ and by $\mathscr{X}_{B,{\tt PGF}}$ the class of $RG$-modules $\mathscr{X}_{B,{\tt PGF}}=\{M\in \textrm{Mod}(RG): \, M\otimes_R B(G,R)\in {\tt PGF}(RG)\}$.
\end{Definition}

\begin{Corollary}\label{Theo1}Let $R$ be a commutative ring such that $\textrm{sfli}(R)<\infty$ and $G$ be an $\textsc{\textbf{lh}}\mathfrak{X}$-group. Then, ${\tt WG Proj}(RG)\subseteq\mathscr{X}_{B,{\tt PGF}}$.
\end{Corollary}

\begin{proof}Since the $RG$-module $B(G,R)$ is $R$-free, this is an immediate consequence of Proposition \ref{Prop1}.\end{proof}

\begin{Proposition}\label{Ppst} Let $R$ be a commutative ring such that $\textrm{sfli}(R)<\infty$ and $G$ be an $\textsc{\textbf{lh}}\mathfrak{X}$-group. Then, 
	$\mathscr{X}_{B,{\tt PGF}}\subseteq {\tt PGF}(RG)$.
\end{Proposition}

\begin{proof}Let $B=B(G,R)$, $\overline{B}=\overline{B}(G,R)$ and consider an $RG$-module $M$ such that the $RG$-module $M\otimes_R B$ is PGF. We also let $V_i=\overline{B}^{\otimes i}\otimes_R B$ for every $i\geq 0$, where $\overline{B}^{\otimes 0}=R$. Since the short exact sequence of $RG$-modules $0\rightarrow R \rightarrow B \rightarrow \overline{B}\rightarrow 0$ is $R$-split, we obtain for every $i\geq 0$ a short exact sequence of $RG$-modules of the form $$0\rightarrow M\otimes_R\overline{B}^{\otimes i}\rightarrow M\otimes_R V_i \rightarrow M\otimes_R \overline{B}^{\otimes i+1}\rightarrow 0.$$ Then, the splicing of the above short exact sequences for every $i\geq 0$ yields an exact sequence of the form
	\begin{equation}\label{Eeq1}
		0\rightarrow M \xrightarrow{\alpha} M\otimes_R V_0 \rightarrow M\otimes_R V_1 \rightarrow M\otimes_R V_2 \rightarrow \cdots.
	\end{equation} 
	Since the $RG$-module $M\otimes_R B$ is PGF and $\overline{B}$ is $R$-projective, we obtain that the $RG$-module $M\otimes_R V_i\cong (M\otimes_R B)\otimes_R \overline{B}^{\otimes i}$ is PGF for every $i\geq 0$, by Proposition \ref{Prop1}. We also consider an $RG$-projective resolution of $M$
	\begin{equation*}
		\textbf{P}=\cdots \rightarrow P_2 \rightarrow P_1 \rightarrow P_0 \xrightarrow{\beta} M \rightarrow 0.
	\end{equation*} 
	Splicing the resolution $\textbf{P}$ with the exact sequence (\ref{Eeq1}), we obtain an acyclic complex of PGF $RG$-modules 
	\begin{equation*}
		\mathfrak{P}=\cdots \rightarrow P_2 \rightarrow P_1 \rightarrow P_0 \xrightarrow{\alpha \beta} M\otimes_R V_0 \rightarrow M\otimes_R V_1 \rightarrow M\otimes_R V_2 \rightarrow \cdots
	\end{equation*} 
	which has syzygy the $RG$-module $M$. It suffices to prove that the complex $I\otimes_{RG}\mathfrak{P}$ is acyclic for every injective $RG$-module $I$. Using \cite[Theorem 5.7]{KS} we will then obtain that the $RG$-module $M$ is PGF. Let $I$ be an injective $RG$-module. Then, the $R$-split short exact sequence of $RG$-modules $0\rightarrow R \rightarrow B \rightarrow \overline{B}\rightarrow 0$ yields an induced exact sequence of $RG$-modules with diagonal action $0\rightarrow I\rightarrow B \otimes_{R}I\rightarrow \overline{B}\otimes_{R} I\rightarrow 0$ which is $RG$-split. Thus, it suffices to prove that the complex $(B\otimes_{R}I)\otimes_{RG}\mathfrak{P}$ is acyclic. Since $B$ is $R$-projective, we obtain that the acyclic complex $\textbf{P}\otimes_R B$ is a projective resolution of the PGF $RG$-module $M\otimes_{R}B$. Therefore, every syzygy module of $\textbf{P}\otimes_R B$ is also a PGF $RG$-module (see \cite[Proposition 2.1]{KS}). Moreover, the $RG$-module $(M\otimes_R B)\otimes_R \overline{B}^{\otimes i}\cong (M\otimes_R \overline{B}^{\otimes i})\otimes_R B$ is PGF for every $i\geq 0$. Consequently, every syzygy module of the acyclic complex
	\begin{equation*}
		\mathfrak{P}\otimes_R B =\cdots\rightarrow  P_1\otimes_R B\rightarrow  P_0\otimes_R B \rightarrow  M\otimes_R V_0 \otimes_R B\rightarrow  M\otimes_R V_1 \otimes_R B\rightarrow \cdots
	\end{equation*} is a PGF $RG$-module. As the functor $\textrm{Tor}^{RG}_1 (I,\_\!\_)$ vanishes on PGF modules, we conclude that the complex $(B\otimes_{R}I)\otimes_{RG}\mathfrak{P}\cong I\otimes_{RG}(\mathfrak{P}\otimes_{R}B)$ is acyclic, as needed.\end{proof}

\begin{Theorem}\label{Cora}Let $R$ be a commutative ring such that $\textrm{sfli}(R)<\infty$ and $G$ be an $\textsc{\textbf{lh}}\mathfrak{X}$-group. Then, $\mathscr{X}_{B,{\tt GProj}}=\mathscr{X}_{B,{\tt PGF}}={\tt PGF}(RG)={\tt WGProj}(RG)={\tt GProj}(RG)$.
\end{Theorem}

\begin{proof}By Corollary \ref{Theo1} and Proposition \ref{Ppst}, we have the inclusions ${\tt WG Proj}(RG)\subseteq\mathscr{X}_{B,{\tt PGF}}\subseteq{\tt PGF}(RG)$. Moreover, ${\tt PGF}(RG)\subseteq {\tt GProj}(RG)$ by \cite[Theorem 4.4]{SS} and the inclusion ${\tt GProj}(RG)$ $\subseteq{\tt WGProj}(RG)$ is clear. We deduce that $\mathscr{X}_{B,{\tt PGF}}={\tt PGF}(RG)={\tt WGProj}(RG)={\tt GProj}(RG)$ and the equality $\mathscr{X}_{B,{\tt GProj}}=\mathscr{X}_{B,{\tt PGF}}$ follows.
\end{proof}

\begin{Remark} \rm According to Theorem \ref{Cora}, it is possible that the equality ${\tt WGProj}(RG)={\tt GProj}(RG)$ holds over a commutative ring $R$ of infinite global dimension 
even if these classes are not equal to the class $\mathscr{X}_{B,{\tt Proj}}=\{M\in \textrm{Mod}(RG): \, M\otimes_R B(G,R)\in {\tt Proj}(RG)\}$ of Benson's cofibrants. However, \cite[Corollary 5.5]{Bis1} and \cite[Theorem 8.4]{St1} yield the equalities $\mathscr{X}_{B,{\tt Proj}}={\tt PGF}(RG)={\tt WGProj}(RG)={\tt GProj}(RG)$ over any commutative ring $R$ of finite global dimension and any $\textsc{\textbf{lh}}\mathfrak{F}$-group or a group of type $\Phi_R$ 
(see also \cite[Conjecture A]{DT}).
\end{Remark}



\begin{Corollary}\label{nionia}Let $R$ be a commutative ring such that $\textrm{sfli}(R)<\infty$ and $G$ be an $\textsc{\textbf{lh}}\mathfrak{X}$-group. Then, ${\tt GProj}(RG)\subseteq {\tt GFlat}(RG)$.
\end{Corollary}

\begin{proof}This is a direct consequence of Theorem \ref{Cora}, since ${\tt PGF}(RG)\subseteq {\tt GFlat}(RG)$.
\end{proof}

\begin{Corollary}Let $R$ be a commutative ring such that $\textrm{sfli}(R)<\infty$ and $G$ be an $\textsc{\textbf{lh}}\mathfrak{X}$-group. Then, for every $RG$-module $M$ we have  $\textrm{Gfd}_{RG}M\leq \textrm{Gpd}_{RG}M =\textrm{PGF-dim}_{RG}M$.
\end{Corollary}

\begin{Corollary}Let $R$ be a commutative ring of finite Gorenstein weak global dimension. Then, Gorenstein projectivity is closed under subgroups that are in $\textsc{\textbf{lh}}\mathfrak{X}$.
\end{Corollary}

\begin{proof}Let $M$ be a Gorenstein projective $RG$-module and $H$ be a subgroup of $G$ such that $H\in\textsc{\textbf{lh}}\mathfrak{X}$. Then, the $RH$-module $\textrm{Res}_H^G(M)$ is weak Gorenstein projective and hence Theorem \ref{Cora} implies that $\textrm{Res}_H^G(M)\in{\tt GProj}(RH)$.\end{proof}

\begin{Remark}\label{rem46}\rm Let $R$ be a commutative ring such that $\textrm{sfli}(R)<\infty$ and $G$ be a group in the class $\mathfrak{X}$, i.e. there exists a weak characteristic module $A$ for $G$ over $R$. We also consider an $RG$-module $M$ such that the $RG$-module $M\otimes_R A$ is PGF. Then, $M$ is a PGF $RG$-module. 
	
	Indeed, there exists an $R$-split $RG$-short exact sequence $0\rightarrow R \rightarrow A \rightarrow \overline{A}\rightarrow 0$, where the $RG$-modules $A,\overline{A}$ are $R$-projectives (this follows from \cite[Theorem 3.14(v)]{St}). Following step by step the proof of Proposition \ref{Ppst}, we construct an acyclic complex of PGF modules 
	\begin{equation*}
		\mathfrak{P'}=\cdots \rightarrow P_2' \rightarrow P_1' \rightarrow P_0' \rightarrow M\otimes_R V_0' \rightarrow M\otimes_R V_1' \rightarrow M\otimes_R V_2' \rightarrow \cdots,
	\end{equation*} where $V_i'={\overline{A}}^{\otimes i}\otimes_R A$, for every $i\geq 0$, and has syzygy the $RG$-module $M$. Using the $R$-split $RG$-short exact sequence $0\rightarrow R \rightarrow A \rightarrow \overline{A} \rightarrow 0 $ and \cite[Theorem 5.7]{KS}, it suffices to show that the complex $I\otimes_{RG}(\mathfrak{P'}\otimes_{R}A)$ is acyclic for every injective $RG$-module $I$. This follows exactly as in the proof of Proposition \ref{Ppst}, since every syzygy module of $\mathfrak{P'}\otimes_{R}A$ is PGF.
\end{Remark}

\begin{Remark}\rm Let $R$ be a commutative ring such that $\textrm{sfli}(R)<\infty$ and $G$ be a group in the class $\mathfrak{X}$. We also consider an $RG$-module $M$ such that the $RG$-module $M\otimes_R A$ is Gorenstein projective. Then, $M$ is a Gorenstein projective $RG$-module. This follows from Remark \ref{rem46} and Theorem \ref{Cora}.
\end{Remark}

The following result provides a generalization of \cite[Theorem 3.11]{Bis1}.
\begin{Corollary} Let $R$ be a commutative ring of finite weak global dimension and $G$ be a group of type $\Phi_R$. Then, $\mathscr{X}_{B,{\tt Proj}}={\tt PGF}(RG)={\tt WProj}(RG)={\tt GProj}(RG)$, where $\mathscr{X}_{B,{\tt Proj}}=\{M\in \textrm{Mod}(RG): \, M\otimes_R B(G,R)\in {\tt Proj}(RG)\}$ is the class of Benson's cofibrants.
\end{Corollary}

\begin{proof}In view of Theorem \ref{Cora}, it suffices to show that $\mathscr{X}_{B,{\tt PGF}}\subseteq\mathscr{X}_{B,{\tt Proj}}$. Let $M\in \mathscr{X}_{B,{\tt PGF}}$. Then, $M\in {\tt WGProj}(RG)\subseteq {\tt WGFlat}(R)$, and hence the finiteness of $\textrm{wgl.dim}(R)$ implies that $M$ is $R$-flat. Since $\textrm{fd}_{RG} B(G,R)<\infty$, we obtain that $\textrm{fd}_{RG}M\otimes_R B(G,R)<\infty$. We conclude that $M\otimes_R B(G,R)\in {\tt PGF}(RG)\cap \overline{{\tt Flat}}(RG)={\tt Proj}(RG)$ (see \cite[Lemma 2.1]{St}).
\end{proof}

\section{Gorenstein injective modules over $\textsc{\textbf{h}}\mathfrak{Y}$-groups}We consider a commutative ring $R$ such that $\textrm{spli}(R)<\infty$ and an $\textsc{\textbf{\textsc{h}}}\mathfrak{Y}$-group $G$. Our goal in this section is to give a characterization of the class ${\tt GInj}(RG)$ in terms of the $RG$-module $B(G,R)$. Moreover, under these conditions, we conclude that the class ${\tt GInj}(RG)$ coincides with the class ${\tt WGInj}(RG)$. 

\begin{Lemma}\label{Illem45}Let $R$ be a commutative ring, $G$ be a group and $H$ be a subgroup of $G$. Then, for every Gorenstein injective $RH$-module $M$, the $RG$-module $\textrm{Coind}^G_H M$ is also Gorenstein injective.
\end{Lemma}

\begin{proof} Let $M$ be a Gorenstein injective $RH$-module. Then, there exists an acyclic complex of injective $RH$-modules $$\textbf{I}=\cdots \rightarrow I_{2}\rightarrow I_1\rightarrow I_0 \rightarrow I_{-1}\rightarrow \cdots$$ such that $M=\textrm{Im}(I_1 \rightarrow I_0)$ and the complex $\textrm{Hom}_{RH}(J,\textbf{I})$ is exact, whenever $J$ is an injective $RH$-module. Thus, the coinduced complex $$\textrm{Coind}^G_H\textbf{I}=\cdots \rightarrow\textrm{Coind}^G_H I_2 \rightarrow\textrm{Coind}^G_H I_1\rightarrow\textrm{Coind}^G_H I_0 \rightarrow\textrm{Coind}^G_H I_{-1}\rightarrow \cdots$$ is an acyclic complex of injective $RG$-modules and has the $RG$-module $\textrm{Coind}^G_H M$ as syzygy. Since every injective $RG$-module $J$ is restricted to an injective $RH$-module, the isomorphism of complexes $\textrm{Hom}_{RG}(J, \textrm{Coind}^G_H\textbf{I})\cong \textrm{Hom}_{RH}(J,\textbf{I})$ implies that the $RG$-module $\textrm{Coind}^G_H M$ is Gorenstein injective.
\end{proof}

\begin{Proposition}\label{Iprop1}Let $R$ be a commutative ring such that $\textrm{spli}(R)<\infty$ and $G$ be an $\textsc{\textbf{h}}\mathfrak{Y}$-group. Consider a weak Gorenstein injective $RG$-module $M$ and an $RG$-module $N$ which is projective as $R$-module. Then, $\textrm{Hom}_{R}(N,M)\in{\tt G Inj}(RG)$.
\end{Proposition}

\begin{proof}We use transfinite induction on the ordinal number $\alpha$, which is such that $G\in \textsc{\textbf{h}}_{\alpha}\mathfrak{Y}$. If $\alpha=0$, then $G$ admits a characteristic module, and hence \cite[Theorem 2.14]{St1} yields $\textrm{spli}(RG) <\infty$. Since $M$ is a weak Gorenstein injective $RG$-module, there exists an acyclic complex of injective $RG$-modules $$\textbf{I}=\cdots \rightarrow I_{2}\rightarrow I_1\rightarrow I_0 \rightarrow I_{-1}\rightarrow \cdots$$ such that $M=\textrm{Im}(I_1 \rightarrow I_0)$. Since $N$ is $R$-projective, we obtain the exact complex of $RG$-injective modules (with diagonal action) $$\textrm{Hom}_R (N,\textbf{I}) = \cdots \rightarrow \textrm{Hom}_R (N,I_1)\rightarrow \textrm{Hom}_R (N,I_0) \rightarrow \textrm{Hom}_R (N,I_{-1})\rightarrow \cdots,$$ where $\textrm{Hom}_R (N,M)= \textrm{Im}(\textrm{Hom}_R (N,I_1) \rightarrow \textrm{Hom}_R (N,I_0))$. Then, the finiteness of $\textrm{spli}(RG)$ implies that the complex $\textrm{Hom}_{RG} (J,\textrm{Hom}_R (N,\textbf{I}) )$ is acyclic for every injective $RG$-module $J$. We conclude that the $RG$-module $\textrm{Hom}_R (N,M)$ is Gorenstein injective.  
	
	Now we assume that $\textrm{Hom}_R (N,M)$ is Gorenstein injective as $RG'$-module for every $\textsc{\textbf{h}}_{\beta}\mathfrak{Y}$-subgroup $G'$ of $G$ and every $\beta<\alpha$. Since $G\in \textsc{\textbf{h}}_{\alpha}\mathfrak{Y}$, there
	exists a $\mathbb{Z}$-split exact sequence of $\mathbb{Z}G$-modules $$0\rightarrow C_r \rightarrow \cdots \rightarrow C_1 \rightarrow C_0 \rightarrow \mathbb{Z} \rightarrow 0,$$	where each $C_i$
	is a direct sum of permutation $\mathbb{Z}G$-modules of the form $\mathbb{Z}[G/G']$, with $G'$ an $\textsc{\textbf{h}}_{\beta}\mathfrak{Y}$-subgroup of $G$ for some $\beta<\alpha$. We note that the integer $r$ is the dimension of the $G$-CW-complex provided by the definition of $G$ being an $\textsc{\textbf{h}}_{\alpha}\mathfrak{Y}$-group. Tensoring the above exact sequence by $R$ we obtain an $R$-split exact sequence of $RG$-modules of the form
	\begin{equation}\label{Ieqqqq}
		0\rightarrow K_r \rightarrow \cdots \rightarrow K_1 \rightarrow K_0 \rightarrow R \rightarrow 0
	\end{equation}
	such that every $K_i$ is a direct sum of modules of the form $R[G/G']$, with $G'$ an $\textsc{\textbf{h}}_{\beta}\mathfrak{Y}$-subgroup of $G$ for some $\beta<\alpha$.
	Applying the functor $\textrm{Hom}_R (\_\!\_ ,\textrm{Hom}_R(N,M))$ to exact sequence (\ref{Ieqqqq}) we obtain the following exact sequence of $RG$-modules:
	\begin{equation}\label{Iieqqqq}	0\rightarrow \textrm{Hom}_R(N,M) \rightarrow \textrm{Hom}_R (K_0 ,\textrm{Hom}_R(N,M)) \rightarrow \cdots \rightarrow \textrm{Hom}_R (K_r ,\textrm{Hom}_R(N,M)) \rightarrow 0.
	\end{equation} 
	Assuming now that $K_i=\oplus_{j_i} R[G/G'_{j_i}]$, we have the following isomorphisms 
	\begin{align*}\textrm{Hom}_R (K_i ,\textrm{Hom}_R(N,M))&\cong \prod_{j_i}\textrm{Hom}_R (R[G/G'_{j_i}] ,\textrm{Hom}_R(N,M))\\
	&\cong \prod_{j_i}\textrm{Hom}_R (RG\otimes_{RG'_{j_i}}R ,\textrm{Hom}_R(N,M))\\
	&\cong \prod_{j_i}\textrm{Hom}_{RG'_{j_i}} (RG,\textrm{Hom}_R(R,\textrm{Hom}_R(N,M)))\\
	&\cong \prod_{j_i}\textrm{Coind}^G_{G'_{j_i}}(\textrm{Hom}_R(N,M))
	\end{align*} for every $i=0,\dots ,r$. 
	From the induction hypothesis and Lemma \ref{Illem45} we deduce that $\textrm{Coind}^G_{G'_{j_i}}(\textrm{Hom}_R(N,M))$ is a Gorenstein injective $RG$-module. Since the class of Gorenstein injective $RG$-modules is closed under arbitrary direct products (see \cite[Theorem 2.6]{H1}), it follows that $\textrm{Hom}_R (K_i ,\textrm{Hom}_R(N,M))$ is a Gorenstein injective $RG$-module for every $i=0,\dots ,r$. Thus, the exact sequence (\ref{Iieqqqq}) yields $\textrm{Gid}_{RG}(\textrm{Hom}_R(N,M))\leq r$. Moreover, $M\in {\tt WGInj}(RG)$, and hence there exists an exact sequence of $RG$-modules of the form $$0\rightarrow M' \rightarrow J_0 \rightarrow J_1 \rightarrow \cdots \rightarrow J_{r-1} \rightarrow M \rightarrow 0,$$ where $J_i$ is injective for every $i=0,1,\dots,r-1$ and $M'\in {\tt WGInj}(RG)$. Since $N$ is $R$-projective, we obtain an exact sequence of $RG$-modules (with diagonal action)
	\begin{equation}\label{Ieqqq}
		0\rightarrow \textrm{Hom}_R(N,M')\rightarrow \textrm{Hom}_R(N,J_0) \rightarrow \cdots \rightarrow \textrm{Hom}_R(N,J_{r-1})\rightarrow \textrm{Hom}_R(N,M) \rightarrow 0,
	\end{equation}where $\textrm{Hom}_R(N,J_i)$ is an injective $RG$-module for every $i=0,1,\dots,r-1$. The same argument as above for the $RG$-module $\textrm{Hom}_R(N,M')$ yields $\textrm{Gid}_{RG}(\textrm{Hom}_R(N,M'))\leq r$ and hence \cite[Theorem 2.22]{H1} implies that $\textrm{Hom}_R(N,M)$ is a Gorenstein injective $RG$-module.\end{proof}
	

\begin{Definition}Let $R$ be a commutative ring and $G$ be a group. We denote by $\mathscr{Y}_{B,{\tt GInj}}$ the class of $RG$-modules $\mathscr{Y}_{B,{\tt GInj}}=\{M\in \textrm{Mod}(RG): \, \textrm{Hom}_R (B(G,R),M)\in {\tt GInj}(RG)\}$.
\end{Definition}

\begin{Corollary}\label{Itheor1}Let $R$ be a commutative ring such that $\textrm{spli}(R)<\infty$ and $G$ be an $\textsc{\textbf{h}}\mathfrak{Y}$-group. Then, ${\tt WG Inj}(RG)\subseteq\mathscr{Y}_{B,{\tt GInj}}$.
\end{Corollary}

\begin{proof}Since the $RG$-module $B(G,R)$ is $R$-free, this is an immediate consequence of Proposition \ref{Iprop1}.
\end{proof}


\begin{Proposition}\label{Itheor2}Let $R$ be a commutative ring such that $\textrm{spli}(R)<\infty$ and $G$ be an $\textsc{\textbf{h}}\mathfrak{Y}$-group. Then, $\mathscr{Y}_{B,{\tt GInj}}\subseteq{\tt GInj}(RG)$.
\end{Proposition}

\begin{proof}Let $B=B(G,R)$, $\overline{B}=\overline{B}(G,R)$ and consider an $RG$-module $M$ such that the $RG$-module $\textrm{Hom}_R (B,M)$ is Gorenstein injective. We also let $V_i=\overline{B}^{\otimes i}\otimes_R B$ for every $i\geq 0$, where $\overline{B}^{\otimes 0}=R$. Since the short exact sequence of $RG$-modules $0\rightarrow R \rightarrow B \rightarrow \overline{B}\rightarrow 0$ is $R$-split, applying the functor $\textrm{Hom}_R (\overline{B}^{\otimes i}\otimes_R \_\!\_, M)$ we obtain for every $i\geq 0$ a short exact sequence of $RG$-modules of the form $$0\rightarrow \textrm{Hom}_R (\overline{B}^{\otimes {i+1}}, M)\rightarrow \textrm{Hom}_R (V_i, M) \rightarrow \textrm{Hom}_R (\overline{B}^{\otimes i}, M)\rightarrow 0.$$ Then, the splicing of the above short exact sequences for every $i\geq 0$ yields an exact sequence of the form
	\begin{equation}\label{IIeq1}
		\cdots \rightarrow \textrm{Hom}_R (V_2, M) \rightarrow \textrm{Hom}_R (V_1, M) \rightarrow \textrm{Hom}_R (V_0, M)\xrightarrow{\alpha} M \rightarrow 0.
	\end{equation}
	Since the $RG$-module $\textrm{Hom}_R (B,M)$ is Gorenstein injective and $\overline{B}$ is $R$-projective, we obtain that the $RG$-module $\textrm{Hom}_R (V_i,M)\cong \textrm{Hom}_R (\overline{B}^{\otimes i},\textrm{Hom}_R (B,M))$ is Gorenstein injective for every $i\geq 0$, by Proposition \ref{Iprop1}. We also consider an $RG$-injective resolution of $M$
	\begin{equation*}
		\textbf{I}=0\rightarrow M \xrightarrow{\beta} I_0 \rightarrow I_1 \rightarrow I_2 \rightarrow \cdots 
	\end{equation*} 
	Splicing the resolution $\textbf{I}$ with the exact sequence (\ref{IIeq1}), we obtain an acyclic complex of Gorenstein injective $RG$-modules 
	\begin{equation*}
		\mathfrak{I}=\cdots \rightarrow \textrm{Hom}_R (V_2, M) \rightarrow \textrm{Hom}_R (V_1, M) \rightarrow \textrm{Hom}_R (V_0, M)\xrightarrow{\beta\alpha} I_0 \rightarrow I_1 \rightarrow I_2 \rightarrow \cdots 
	\end{equation*} 
	which has syzygy the $RG$-module $M$. It suffices to prove that the complex $\textrm{Hom}_{RG}(J,\mathfrak{I})$ is acyclic for every injective $RG$-module $J$. Using \cite[Theorem 2.1]{Bouch} we will then obtain that the $RG$-module $M$ is Gorenstein injective. Let $J$ be an injective $RG$-module. Then, the $R$-split short exact sequence of $RG$-modules $0\rightarrow R \rightarrow B \rightarrow \overline{B}\rightarrow 0$ yields an induced exact sequence of $RG$-modules with diagonal action $0\rightarrow J\rightarrow B \otimes_{R}J\rightarrow \overline{B}\otimes_{R} J\rightarrow 0$ which is $RG$-split. Thus, it suffices to prove that the complex $\textrm{Hom}_{RG}(B\otimes_{R} J,\mathfrak{I})\cong \textrm{Hom}_{RG}(J,\textrm{Hom}_R (B,\mathfrak{I}))$ is acyclic. Since $B$ is $R$-projective, we obtain that the acyclic complex $\textrm{Hom}_R(B,\mathbf{I})$ is an injective resolution of the Gorenstein injective $RG$-module $\textrm{Hom}_R(B,M)$. It follows that every syzygy module of $\textrm{Hom}_R(B,\mathbf{I})$ is also a Gorenstein injective $RG$-module (see \cite[Theorem 2.6]{H1}). 
	 Moreover, the $RG$-module $\textrm{Hom}_R(B,\textrm{Hom}_R(\overline{B}^i,M))\cong \textrm{Hom}_R(V_i,M)$ is Gorenstein injective for every $i\geq 0$. Consequently, every syzygy module of the acyclic complex
	\begin{equation*}
		\textrm{Hom}_R (B,\mathfrak{I})=\cdots\rightarrow \textrm{Hom}_R(B, \textrm{Hom}_R(V_0,M)) \rightarrow  \textrm{Hom}_R(B,I_0) \rightarrow  \textrm{Hom}_R(B,I_1)\rightarrow \cdots
	\end{equation*} is a Gorenstein injective $RG$-module. As the functor $\textrm{Ext}^1_{RG} (J,\_\!\_)$ vanishes on Gorenstein injective $RG$-modules, we conclude that the complex $\textrm{Hom}_{RG}(J,\textrm{Hom}_R (B,\mathfrak{I}))$ is acyclic, as needed.\end{proof}

\begin{Theorem}\label{Iicora}Let $R$ be a commutative ring such that $\textrm{spli}(R)<\infty$ and $G$ be an $\textsc{\textbf{h}}\mathfrak{Y}$-group. Then, $\mathscr{Y}_{B,{\tt GInj}}={\tt GInj}(RG)={\tt WGInj}(RG)$.
\end{Theorem}


\begin{proof}By Corollary \ref{Itheor1}, we have ${\tt WG Inj}(RG)\subseteq\mathscr{Y}_{B,{\tt GInj}}$. Moreover, Proposition \ref{Itheor2} yields $\mathscr{Y}_{B,{\tt GInj}}\subseteq{\tt GInj}(RG)$ and the inclusion ${\tt GInj}(RG)\subseteq{\tt WGInj}(RG)$ is clear. We conclude that $\mathscr{Y}_{B,{\tt GInj}}={\tt GInj}(RG)={\tt WGInj}(RG)$, as needed.
\end{proof}

\begin{Corollary}Let $R$ be a commutative ring of finite Gorenstein global dimension. Then, Gorenstein injectivity is closed under subgroups that are in $\textsc{\textbf{h}}\mathfrak{Y}$.
\end{Corollary}

\begin{proof}Let $M$ be a Gorenstein injective $RG$-module and $H$ be a subgroup of $G$ such that $H\in\textsc{\textbf{h}}\mathfrak{Y}$. Then, the $RH$-module $\textrm{Res}_H^G(M)$ is weak Gorenstein injective and hence Theorem \ref{Iicora} implies that $\textrm{Res}_H^G(M)\in{\tt GInj}(RH)$.\end{proof}

\begin{Remark}\label{rem58}\rm Let $R$ be a commutative ring such that $\textrm{spli}(R)<\infty$ and $G$ be a group in the class $\mathfrak{Y}$, i.e. there exists a characteristic module $A$ for $G$ over $R$.  We also consider an $RG$-module $M$ such that the $RG$-module $\textrm{Hom}_R(A,M)$ is Gorenstein injective. Then, $M$ is a Gorenstein injective $RG$-module. 
	
	Indeed, there exists an $R$-split $RG$-short exact sequence $0\rightarrow R \rightarrow A \rightarrow \overline{A}\rightarrow 0$, where the $RG$-modules $A,\overline{A}$ are $R$-projective. Following step by step the proof of Proposition \ref{Itheor2}, we construct an acyclic complex of Gorenstein injective modules
		\begin{equation*}
			\mathfrak{I'}=\cdots \rightarrow \textrm{Hom}_R (V_2', M) \rightarrow \textrm{Hom}_R (V_1', M) \rightarrow \textrm{Hom}_R (V_0', M)\rightarrow I_0' \rightarrow I_1' \rightarrow I_2' \rightarrow \cdots,
		\end{equation*}
 where $V_i'={\overline{A}}^{\otimes i}\otimes_R A$, for every $i\geq 0$, which has syzygy the $RG$-module $M$. Using the $R$-split $RG$-short exact sequence $0\rightarrow R \rightarrow A \rightarrow \overline{A} \rightarrow 0 $ and \cite[Theorem 2.1]{Bouch}, it suffices to show that the complex $\textrm{Hom}_{RG}(J,\textrm{Hom}_R (A,\mathfrak{I'}))$ is acyclic for every injective $RG$-module $J$. This follows exactly as in the proof of Proposition \ref{Itheor2}, since every syzygy module of $\textrm{Hom}_R (A,\mathfrak{I'})$ is Gorenstein injective.
\end{Remark}

\begin{Remark}\rm The results of this section require the group $G$ to be fully (not locally) in the hierarchy. The reason is that we do not know weather there exists a Gorenstein injective analogue of \cite[Lemma 5.6]{BC} or \cite[Proposition 4.5]{St1}, and hence we are not able to generalize Proposition \ref{Iprop1} over any $\textsc{\textbf{lh}}\mathfrak{Y}$-group.
\end{Remark}


\section{Gorenstein homological dimension of $\textsc{\textbf{lh}}\mathfrak{F}$-groups}
Our goal in this section is to determine the Gorenstein homological dimension $\textrm{Ghd}_R G$ of an $\textsc{\textbf{lh}}\mathfrak{F}$-group $G$ over a commutative ring of finite Gorenstein weak global dimension.

\begin{Definition}Let $R$ be a commutative ring and $G$ be a group.
	
	$\textrm{f.k}(RG):=\textrm{sup}\,\{\textrm{fd}_{RG}M \, : \, M\in \textrm{Mod}(RG), \, \textrm{fd}_{RH}M<\infty \, \textrm{for every finite} \,\, H\leq G\}$.
	
	$\textrm{fin.f.dim}(RG):=\textrm{sup}\,\{\textrm{fd}_{RG}M \, : \, M\in \textrm{Mod}(RG), \, \textrm{fd}_{RG}M<\infty\}$.
\end{Definition}

\begin{Lemma}\label{l220}Let $R$ be a commutative ring and $G$ be a group. Then, for every subgroup $H$ of $G$ we have $\textrm{fin.f.dim}(RH)\leq \textrm{fin.f.dim}(RG)$.
\end{Lemma}

\begin{proof}It suffices to assume that $\textrm{fin.f.dim}(RG)=n<\infty$. Let $M$ be an $RH$-module such that $\textrm{fd}_{RH}M=k<\infty$. Then, there exists an $RH$-flat resolution of $M$ of length $k$ $$0\rightarrow F_k \rightarrow \cdots \rightarrow F_1 \rightarrow F_0 \rightarrow M \rightarrow 0$$ and hence we obtain an exact sequence of $RG$-modules of the form $$0\rightarrow \textrm{Ind}^G_H F_k \rightarrow \cdots \rightarrow \textrm{Ind}^G_H F_1 \rightarrow \textrm{Ind}^G_H F_0 \rightarrow \textrm{Ind}^G_H M\rightarrow 0$$ which constitutes an $RG$-flat resolution of $\textrm{Ind}^G_H M$ of length $k$. Since M is isomorphic to a direct summand of $\textrm{Res}^G_H (\textrm{Ind}^G_H M)$, we obtain that $\textrm{fd}_{RG}\textrm{Ind}^G_H M=k$. Thus, $\textrm{fd}_{RH}M=k\leq n$ for every $RH$-module $M$ of finite flat dimension. We conclude that $\textrm{fin.f.dim}(RH)\leq \textrm{fin.f.dim}(RG)$, as needed.
\end{proof}

\begin{Proposition}\label{prop224}Let $R$ be a commutative ring and $G$ be an $\textsc{\textbf{lh}}\mathfrak{F}$-group. Then, $\textrm{f.k}(RG)\leq \textrm{fin.f.dim}(RG)$.
\end{Proposition}

\begin{proof}It suffices to assume that $\textrm{fin.f.dim}(RG)=n<\infty$. Let $M$ be an $RG$-module such that $\textrm{fd}_{RF}M<\infty$ for every finite subgroup $F$ of $G$. We will first show that $\textrm{fd}_{RH}M\leq n$ over any $\textsc{\textbf{h}}\mathfrak{F}$-subgroup $H$ of $G$. We use transfinite induction on the ordinal number $\alpha$, which is such that $H\in \textsc{\textbf{h}}_{\alpha}\mathfrak{F}$. If $\alpha=0$, then $H$ is finite and hence $\textrm{fd}_{RH}M<\infty$. Then, Lemma \ref{l220} yields $\textrm{fd}_{RH}M\leq\textrm{fin.f.dim}(RH)\leq \textrm{fin.f.dim}(RG)= n$. Now we assume that $\textrm{fd}_{RH'}M\leq n$ for every $\textsc{\textbf{h}}_{\beta}\mathfrak{F}$-subgroup $H'$ of $G$ and every $\beta<\alpha$. Let $H$ be an $\textsc{\textbf{h}}_{\alpha}\mathfrak{F}$-subgroup of $G$. Then, there
	exists an exact sequence of $\mathbb{Z}H$-modules $$0\rightarrow C_r \rightarrow \cdots \rightarrow C_1 \rightarrow C_0 \rightarrow \mathbb{Z} \rightarrow 0,$$	where each $C_i$
	is a direct sum of permutation $\mathbb{Z}H$-modules of the form $\mathbb{Z}[H/H']$, with $H'$ an $\textsc{\textbf{h}}_{\beta}\mathfrak{F}$-subgroup of $H$ for some $\beta<\alpha$. We note that the integer $r$ is the dimension of the $H$-CW-complex provided by the definition of $H$ being an $\textsc{\textbf{h}}_{\alpha}\mathfrak{F}$-group. The above exact sequence yields an exact sequence of $RH$-modules
	\begin{equation}\label{eq14}
		0\rightarrow M_r \rightarrow \cdots \rightarrow M_1 \rightarrow M_0 \rightarrow M \rightarrow 0,
	\end{equation}
	where each $M_i$ is a direct sum of modules of the form ${\textrm{Ind}^H_{H'}}{\textrm{Res}^H_{H'}} M$, where $H'\in \textsc{\textbf{h}}_{\beta}\mathfrak{F}$, $\beta<\alpha$ (see also \cite[Lemma 2.3]{Bis2}). Our induction hypothesis implies that $\textrm{fd}_{RH'}\textrm{Res}^H_{H'}M\leq n$ for every $H'\in \textsc{\textbf{h}}_{\beta}\mathfrak{F}$, $\beta<\alpha$, and hence we also have $\textrm{fd}_{RH}\textrm{Ind}^H_{H'}\textrm{Res}^H_{H'}M\leq n$ for every $H'\in \textsc{\textbf{h}}_{\beta}\mathfrak{F}$, $\beta<\alpha$. Consequently, $\textrm{fd}_{RH}M_i <\infty$ for every $i=0,\dots,r$. Thus, the exact sequence (\ref{eq14}) yields $\textrm{fd}_{RH}M <\infty$. Invoking Lemma \ref{l220}, we infer that $\textrm{fd}_{RH}M\leq \textrm{fin.f.dim}(RH)\leq \textrm{fin.f.dim}(RG)= n$.
	
	Let $G$ be an $\textsc{\textbf{lh}}\mathfrak{F}$-group. Then, $G$ can be expressed as the filtered union of its finitely generated subgroups $(G_{\lambda})_{\lambda}$, which are all contained in $\textsc{\textbf{h}}\mathfrak{F}$. Since $G_{\lambda}\in \textsc{\textbf{h}}\mathfrak{F}$, we have $\textrm{fd}_{RG_{\lambda}}M\leq n$. We consider an exact sequence of $RG$-modules 
	\begin{equation}\label{eq15}
		0 \rightarrow K_n \rightarrow F_{n-1} \rightarrow \cdots \rightarrow F_1 \rightarrow F_0 \rightarrow M \rightarrow 0,
	\end{equation}
	where the $RG$-module $F_i$ is flat for every $i=0,\dots,n-1$. Then, $K_n$ is a flat $RG_{\lambda}$-module, and hence the $RG$-module $\textrm{Ind}^G_{G_{\lambda}}K_n$ is also flat for every $\lambda$. Consequently, the $RG$-module $K_n\cong  {\lim\limits_{\longrightarrow}}_{\lambda}\textrm{Ind}^G_{G_{\lambda}}K_n$ is flat as direct limit of flat modules. Thus, the exact sequence (\ref{eq15}) yields $\textrm{fd}_{RG}M\leq n$. We conclude that $\textrm{f.k}(RG)\leq n$, as needed.
\end{proof}


\begin{Lemma}\label{prop225}Let $R$ be a commutative ring such that $\textrm{sfli}(R)<\infty$ and $G$ be a group. Then, $\textrm{sfli}(RG)\leq \textrm{f.k}(RG)$.
\end{Lemma}

\begin{proof}It suffices to assume that $\textrm{f.k}(RG)=n<\infty$. Let $I$ be an injective $RG$-module and $H$ a finite subgroup of $G$. Then, $\textrm{Ghd}_R H=0$ (see \cite[Proposition 3.6]{St1}), and hence \cite[Proposition 3.13]{St} yields $\textrm{sfli}(RH)\leq \textrm{sfli}(R) <\infty$. Since $I$ is injective as $RH$-module, we obtain that $\textrm{fd}_{RH}I<\infty$. It follows that $\textrm{fd}_{RG}I\leq \textrm{f.k}(RG)=n $ for every injective $RG$-module $I$. We conclude that $\textrm{sfli}(RG)\leq \textrm{f.k}(RG)$, as needed.
\end{proof}

\begin{Corollary}\label{cor226}Let $R$ be a commutative ring such that $\textrm{sfli}(R)<\infty$ and $G$ be an $\textsc{\textbf{lh}}\mathfrak{F}$-group. Then, $\textrm{f.k}(RG)=\textrm{sfli}(RG)=\textrm{fin.f.dim}(RG)$.
\end{Corollary}

\begin{proof}Since $RG\cong {(RG)}^{\textrm{op}}$, by \cite[Proposition 2.4(i)]{Emta} we obtain that $\textrm{fin.f.dim}(RG)\leq \textrm{sfli}(RG)$. By Proposition \ref{prop224}, we have $\textrm{f.k}(RG)\leq \textrm{fin.f.dim}(RG)$. Moreover, Lemma \ref{prop225} yields $\textrm{sfli}(RG)\leq \textrm{f.k}(RG)$. We conclude that $\textrm{f.k}(RG)=\textrm{sfli}(RG)=\textrm{fin.f.dim}(RG)$, as needed. 
\end{proof}


\begin{Remark}\label{rem711} \rm Since the $RG$-module $B(G,R)$ is $R$-free and admits an $R$-split $RG$-linear monomorphism $\iota: R \rightarrow B(G,R)$, we infer that $B(G,R)$ is a weak characteristic module for $G$ over $R$ if and only if $\textrm{fd}_{RG} B(G,R)<\infty$.
\end{Remark}

The following result provides the Gorenstein flat analogue of \cite[Theorem 7.6]{St1} (see also \cite[Theorem A.1]{ET2}).

\begin{Theorem}\label{theo712}Let $R$ be a commutative ring such that $\textrm{sfli}(R)<\infty$ and consider an $\textsc{\textbf{lh}}\mathfrak{F}$-group $G$. Then:
	\begin{itemize}
		\item[(i)] $B(G,R)$ is a weak characteristic module for $G$ if and only if ${{\textrm{Ghd}}_{R}G}<\infty$,
		\item[(ii)] ${{\textrm{Ghd}}_{R}G}=\textrm{fd}_{RG} B(G,R)$.
	\end{itemize}
\end{Theorem}

\begin{proof}(i) If $B(G,R)$ is a weak characteristic module, then \cite[Theorem 3.14]{St} implies that ${\textrm{Ghd}}_{R}G<\infty$. Conversely, we assume that ${\textrm{Ghd}}_{R}G<\infty$. Then, Corollary \ref{cor226} yields $\textrm{f.k}(RG)=\textrm{sfli}(RG)\leq {\textrm{Ghd}}_{R}G + \textrm{sfli}(R) <\infty$ (see \cite[Proposition 3.13]{St}). Since $B(G,R)$ is free as $RH$-module for every finite subgroup $H$ of $G$, we infer that $\textrm{fd}_{RG}B(G,R)\leq \textrm{f.k}(RG)<\infty$. We conclude that $B(G,R)$ is a weak characteristic module for $G$ over $R$ (see Remark \ref{rem711}).
	
	(ii) Using (i) and Remark \ref{rem711}, we have ${{\textrm{Ghd}}_{R}G}=\infty$ if and only if $\textrm{fd}_{RG}B(G,R)=\infty$. If ${{\textrm{Ghd}}_{R}G}<\infty$, then (i) implies that $B(G,R)$ is a weak characteristic module for $G$ over $R$, and hence, in view of \cite[Corollary 3.7(i),(iii))]{St}, we conclude that ${{\textrm{Ghd}}_{R}G}=\textrm{fd}_{RG} B(G,R)$.
\end{proof}

\begin{Question} Is it true that ${{\textrm{Ghd}}_{R}G}=\textrm{fd}_{RG} B(G,R)$ for every commutative ring $R$ such that $\textrm{sfli}(R)<\infty$ and every group $G$?
\end{Question}

\begin{Remark}\rm (i) Let $R$ be a commutative ring, $G$ be a group and  	$$\textrm{f.k}_{\mathfrak{X}}(RG):=\textrm{sup}\,\{\textrm{fd}_{RG}M \, : \, M\in \textrm{Mod}(RG), \, \textrm{fd}_{RH}M<\infty \,\, \textrm{for every} \,\, H\in \mathfrak{X},\, H\leq G\}.$$ Then, Proposition \ref{prop224}, Lemma \ref{prop225} and Corollary \ref{cor226} continue to hold if we assume that $G$ is an $\textsc{\textbf{lh}}\mathfrak{X}$-group and replace the invariant $\textrm{f.k}(RG)$ by $\textrm{f.k}_{\mathfrak{X}}(RG)$, with similar proofs (one must also apply \cite[Theorem 3.14]{St} in the case of Lemma \ref{prop225}). Consequently, for every commutative ring $R$ such that $\textrm{sfli}(R)<\infty$ and every $\textsc{\textbf{lh}}\mathfrak{X}$-group $G$ we have $$\textrm{f.k}_{\mathfrak{X}}(RG)=\textrm{sfli}(RG)=\textrm{fin.f.dim}(RG).$$
Moreover, it follows from \cite[Corollary 3.5]{St1} and \cite[Theorem 3.14]{St} that every finite group $H$ belongs to the class $\mathfrak{X}$ and hence $\textrm{f.k}_{\mathfrak{X}}(RG)\leq \textrm{f.k}(RG)$. Thus, it is unclear whether the proof of Theorem \ref{theo712} can be generalized in the same way.

(ii) Similarly, for every commutative ring $R$ and every group $G$ we define the invariant 
$$\textrm{k}_{\mathfrak{Y}}(RG):=\textrm{sup}\,\{\textrm{pd}_{RG}M \, : \, M\in \textrm{Mod}(RG), \, \textrm{pd}_{RH}M<\infty \,\, \textrm{for every} \,\, H\in \mathfrak{Y},\, H\leq G\}.$$ Then, \cite[Lemma 7.2, Lemma 7.3 and Proposition 7.4]{St1} continue to hold if we assume that $G$ is an $\textsc{\textbf{lh}}\mathfrak{Y}$-group and replace the invariant $\textrm{k}(RG)$ by $\textrm{k}_{\mathfrak{Y}}(RG)$, with similar proofs (one must also apply \cite[Theorem 2.14]{St1} in the case of \cite[Lemma 7.3]{St1}). Consequently, for every commutative ring $R$ such that $\textrm{spli}(R)<\infty$ and every $\textsc{\textbf{lh}}\mathfrak{Y}$-group $G$ we have $$\textrm{k}_{\mathfrak{Y}}(RG)=\textrm{spli}(RG)=\textrm{silp}(RG)=\textrm{fin.dim}(RG).$$

 \end{Remark}

\begin{Remark}\rm \label{rem78}Let $R$ be a commutative ring and $G$ be a group such that $\textrm{fd}_{\mathbb{Z}G}B(G,\mathbb{Z})<\infty$. Then $\textrm{fd}_{RG}B(G,R)\leq \textrm{fd}_{\mathbb{Z}G}B(G,\mathbb{Z})<\infty$. Indeed, let $\textrm{fd}_{\mathbb{Z}G}B(G,\mathbb{Z})=n$ and consider a $\mathbb{Z}G$-flat resolution $$0\rightarrow F_n \rightarrow F_{n-1}\rightarrow \cdots \rightarrow F_0 \rightarrow B(G,\mathbb{Z}) \rightarrow 0$$ of $B(G,\mathbb{Z})$. Since $B(G,\mathbb{Z})$ is $\mathbb{Z}$-free (and hence $\mathbb{Z}$-flat), the above exact sequence is $\mathbb{Z}$-pure. Thus, we obtain an exact sequence of $RG$-modules $$0\rightarrow F_n\otimes_{\mathbb{Z}}R \rightarrow F_{n-1}\otimes_{\mathbb{Z}}R \rightarrow \cdots \rightarrow  F_0\otimes_{\mathbb{Z}}R \rightarrow B(G,\mathbb{Z})\otimes_{\mathbb{Z}}R =B(G,R) \rightarrow 0$$ which constitutes an $RG$-flat resolution of $B(G,R)$. Hence, $\textrm{fd}_{RG}B(G,R)\leq \textrm{fd}_{\mathbb{Z}G}B(G,\mathbb{Z})$. 
\end{Remark}

\begin{Corollary}Let $R$ be a commutative ring such that $\textrm{sfli}(R)<\infty$ and $G$ be an $\textsc{\textbf{lh}}\mathfrak{F}$-group of type FP$_{\infty}$. Then, ${{\textrm{Ghd}}_{R}G}=\textrm{fd}_{RG} B(G,R)<\infty$.
\end{Corollary}

\begin{proof}The equality ${{\textrm{Ghd}}_{R}G}=\textrm{fd}_{RG} B(G,R)$ follows from Theorem \ref{theo712}. Since the $\textsc{\textbf{lh}}\mathfrak{F}$-group is of type FP$_{\infty}$, using \cite[Corollary B.2(2)]{Kr2}, which is also valid for $\textsc{\textbf{lh}}\mathfrak{F}$-groups, we infer that $\textrm{fd}_{\mathbb{Z}G} B(G,\mathbb{Z})<\infty$. Then, $\textrm{fd}_{RG}B(G,R)\leq \textrm{fd}_{\mathbb{Z}G}B(G,\mathbb{Z})< \infty$ (see Remark \ref{rem78}).
\end{proof}

\begin{Corollary}Let $G$ be an $\textsc{\textbf{lh}}\mathfrak{F}$-group of type FP$_{\infty}$. Then, ${{\textrm{Ghd}}_{\mathbb{Z}}G}=\textrm{fd}_{\mathbb{Z}G} B(G,\mathbb{Z})<\infty$.
\end{Corollary}

\begin{Corollary}Let $R$ be a commutative ring such that $\textrm{sfli}(R)<\infty$ and $G$ be an $\textsc{\textbf{lh}}\mathfrak{F}$-group of type FP$_{\infty}$. Then, $\textrm{f.k}(RG)=\textrm{sfli}(RG)=\textrm{fin.f.dim}(RG)<\infty$. In particular, if $M$ is an $RG$-module, then $\textrm{fd}_{RG}M<\infty$ if and only if $\textrm{fd}_{RH}M<\infty$ for every finite subgroup $H$ of $G$.
\end{Corollary}

\begin{proof}In view of Corollary \ref{cor226}, it suffices to prove that $\textrm{sfli}(RG)<\infty$. From \cite[Corollary B.2(2)]{Kr2}, which is also valid for $\textsc{\textbf{lh}}\mathfrak{F}$-groups, and Remarks \ref{rem711}, \ref{rem78}, we obtain that $B(G,R)$ is a weak characteristic module for $G$ over $R$. Thus, \cite[Theorem 3.14]{St} yields $\textrm{sfli}(RG)<\infty$, as needed.\end{proof}

\textit{Acknowledgement.} Research supported by the Hellenic Foundation for Research and Innovation (H.F.R.I.) under the “1st Call for H.F.R.I. Research Projects to support Faculty members and Researchers and the procurement of high-cost research equipment grant”, project 
number 4226. The author wishes to thank the anonymous referees for reading the paper carefully and providing several useful comments and suggestions.

\vspace{0.05in}

{\small {\sc Department of Mathematics,
             University of Athens,
             Athens 15784,
             Greece}}

{\em E-mail address:} {\tt dstergiop@math.uoa.gr}

\end{document}